\documentclass[10pt]{amsart}
\pdfoutput=1
\PassOptionsToPackage{numbers,square}{natbib}
\usepackage{aiproofs}
\usepackage{needspace}

\newcommand{\reals}{\mathbb{R}}

\newcommand{\Cc}{\mathcal{C}}
\newcommand{\Dc}{\mathcal{D}}

\def\sint{\begingroup\textstyle\int\endgroup} 

\DeclareMathOperator{\trace}{tr}
\DeclareMathOperator{\diag}{diag}

\newtheorem{theorem}{Theorem}[section]

\newtheorem{prop}[theorem]{Proposition}
\newtheorem{lemma}[theorem]{Lemma}

\theoremstyle{definition}

\definecolor{cdarkred}{RGB}{180,70,30} 
\definecolor{cdarkgreen}{RGB}{0,130,0} 
\definecolor{azure}{RGB}{240,255,255}
\definecolor{cmidgreen}{RGB}{201,219,216}
\definecolor{clightgreen}{RGB}{200,255,200}
\definecolor{clightgray}{RGB}{240,240,240}
\definecolor{cmygreen}{RGB}{10,120,0}
\definecolor{cmyblue}{RGB}{50,50,255}
\definecolor{sblue}{RGB}{100,100,255}
\definecolor{cmypurple}{RGB}{150,50,180}
\definecolor{cmyred}{RGB}{238, 64,  0}
\definecolor{cdarkblue}{RGB}{30,75,170}
\definecolor{csharpred}{RGB}{255,0,0}

\newtcolorbox{dashedbox}[1][]{
  nobeforeafter,
  colback=yellow!10!white,
  boxrule=0pt,
  enhanced jigsaw,
  top=2pt,
  left=2pt,
  right=2pt,
  bottom=0pt,
  borderline horizontal={0.5pt}{0pt}{dashed},
  borderline vertical={0.5pt}{0pt}{dashed},
  #1
}

\newtcbox{\tcbmathbox}[1][]{
  nobeforeafter,
  colback=yellow!30!white,
  boxrule=0pt,
  enhanced jigsaw,
  top=0pt,
  left=1pt,
  right=1pt,
  bottom=0pt,
  borderline horizontal={0.5pt}{0pt}{dashed},
  borderline vertical={0.5pt}{0pt}{dashed},
  #1
}

\newcommand{\Ec}{\mathcal E}
\newcommand{\Tc}{\mathcal T}
\newcommand{\eps}{\varepsilon}

\newcommand{\tp}{\mathsf T}
\newcommand{\one}{\mathbf 1}
\theoremstyle{definition}
\newtheorem{algo}[theorem]{Algorithm}
\newcommand{\repourl}{\texttt{https://github.com/suvrit/komlos}}
\AtBeginEnvironment{thebibliography}{\small}
\numberwithin{table}{section}
\makeatletter
\g@addto@macro\normalsize{%
  \abovedisplayskip=6pt plus 2pt minus 1pt
  \belowdisplayskip=\abovedisplayskip
  \abovedisplayshortskip=3pt plus 1pt
  \belowdisplayshortskip=4pt plus 1pt minus 1pt
  \jot=3pt}
\makeatother
\hypersetup{
  pdftitle={Tighter bounds on Komlos discrepancy: existence and algorithmic results},
  pdfauthor={Emrullah Akbas, Suvrit Sra}}

\begin{document}
\title[Tighter bounds on Koml\'os discrepancy]{%
  Tighter bounds on Koml\'os discrepancy: \\ existence and algorithmic results}
\author[Emrullah Akbas]{Emrullah Akbas\textsuperscript{\normalfont*}}
\author[Suvrit Sra]{Suvrit Sra\textsuperscript{\normalfont*}}
\thanks{*\,TU Munich, Dept.~of Mathematics, School of CIT; Garching, Germany.
  Emails: \texttt{emrullah.akbas@tum.de}, \texttt{s.sra@tum.de}.}
\date{}
\begin{abstract}
Guo, Fang, and Lu recently proved the Koml\'os conjecture: for vectors
$v_1,\ldots,v_n\in\reals^d$ of Euclidean norm at most one, there are signs
$\eps_j\in\{-1,1\}$ with $\|\sum_j\eps_jv_j\|_\infty\le3\sqrt{2\pi}$.
We give a short proof of the bound $3\pi$ that keeps the geometric
lifting framework of \cite{guo2026vector}  while replacing  the analytic core by a quadratic Dirichlet
energy. Moreover, using a more fine-grained analysis of the stability of the product-cosine function of \cite{smirnov2026discrepancy} under translations, we further sharpen the bound to below \(6.9013\).
On the algorithmic side,  based on the polynomial-time construction of \cite{guo2026poly} we give a deterministic algorithm that finds a coloring  of discrepancy at most $37.54$ using at most $\widetilde O(mn+n^4)$ arithmetic operations. 
\end{abstract}
\maketitle
\raggedbottom
\setlength{\parskip}{3pt plus 1pt minus .5pt}

\section{Introduction}\label{sec:intro}
The Koml\'os problem (erstwhile conjecture) asks for an absolute constant $C$ such that for all vectors $v_1,\ldots,v_n\in\reals^d$ with $\|v_j\|_2\le1$ there are
signs $\eps_j\in\{-1,1\}$ with $\|\sum_j\eps_jv_j\|_\infty\le C$.
Guo, Fang, and Lu (GFL)~\citep{guo2026vector} recently resolved Koml\'os conjecture  with
$C=3\sqrt{2\pi}$. They use a convex-body transform and signing recursion
that originate with Banaszczyk~\citep{banaszczyk1998}, and a
cosine/Fisher initialization originating with Smirnov and
Vershynin~\citep{smirnov2026discrepancy}. We refer to other nice expositions of their
proof by~\citep{bandeira2026komlos,karingula2026elementaryproofkomlosconjecture}.

The first part of this paper studies the existence of $C$. We retain the
geometric lifting framework of~\citep{guo2026vector} but replace its
nonlinear analytic implementation by an argument based on a quadratic
Dirichlet energy.
The result is a short, expository proof that gives the looser bound $C=3\pi$
(Theorem~\ref{thm:main}); we present it first because it is the simplest
form of our  argument and introduces the objects that the sharper argument
refines. Specifically, we study an energy argument that controls the overlap
between a density and its translates through a single derivative
estimate, and this estimate is what determines the discrepancy bound. In particular, using  a more fine-grained analysis of the stability of the overlap  between the product-cosine density and its translates (rather than crudely estimating it  by the energy argument),
we obtain the tighter bound $C\le 6.9013$
(Theorem~\ref{thm:cosine} in Section~\ref{sec:cosine}), which improves
on $3\sqrt{2\pi}\approx7.52$ of~\citep{guo2026vector}.

The second part of this paper studies the algorithmic side of Koml\'os. 
Guo, Fang, and
Lu~\citep{guo2026poly} gave the first polynomial-time  algorithm for the Koml\'os problem, with a discrepancy bound $C=8892$ and a
running time of $\widetilde O(n^10)$. The first version of this
paper~\citep{akbas2026v1} and, announced in parallel,
Li~\citep{li2026fast} obtained an algorithm with improved complexity and sharper constants; Li did so by screening rows, by a clock that schedules
when a row must be inspected, and by a coupled estimate for the Gram
factor and the low-spectrum resolvent, and Sections~\ref{sec:geometry}
and~\ref{sec:screening} build on all three. Section~\ref{sec:related}
records the chronology and the bounds.

Starting from this spectral framework, we give a deterministic
algorithm that returns a signing with $\|A\eps\|_\infty<37.54$ using
$O(mn)+\widetilde O(n^4)$ arithmetic operations and no fast matrix
multiplication (Theorem~\ref{thm:algo}); up to logarithmic factors, this
cost is below Li's $n^{\omega+2}$ for every attainable exponent
$\omega>2$. The walk moves a fractional signing to the
boundary of the cube while barriers protect every row, and it must keep
enough free directions to do so. Our analysis counts the row constraints
and the positive curvature of a spectral potential against one budget,
and this joint count is what improves the constant. Finite-step
estimates then bound the number of moves, and certified row screening
bounds their cost, so that the work beyond one scan of the input does
not depend on $m$. The proved step sizes are far too small for practice,
so Section~\ref{sec:practical} adds a routine that tries a few random
signings and returns one only when an outward-rounded interval
computation certifies its discrepancy; Appendix~\ref{app:experiments}
compares it with the deterministic walk on matched inputs.

Section~\ref{sec:existence} proves Theorems~\ref{thm:main}
and~\ref{thm:cosine}, with regularity details and the numerical
certificate in Appendix~\ref{app:existence}. Section~\ref{sec:algo-intro}
states Theorem~\ref{thm:algo}, explains the construction, and proves its
discrepancy bound, with the deferred verifications and the
finite-precision and finite-arithmetic details in
Appendices~\ref{app:walk}--\ref{app:arithmetic};
Section~\ref{sec:screening} proves its running time; and
Section~\ref{sec:practical} gives the checked routine and discusses
limitations and open problems. The code, the experiment records, and the exact rational
certificates behind every numerical inequality in the paper are available
in the repository~\repourl.

\subsection{Related work}\label{sec:related}
Koml\'os conjecture itself dates back to the 1980s. Early works by \citep{beck1981roth} and Spencer~\citep{spencer1985six} were state-of-the-art for almost three decades. Banaszczyk~\citep{banaszczyk1998}
proved the bound $O(\sqrt{\log n})$ through studying the Gaussian measure of
convex bodies; his backward recursion of convex sets is the framework underlying the recent breakthrough work 
of Guo, Fang, and Lu
~\citep{guo2026vector}; Karingula and
Lovett~\citep{karingula2026elementaryproofkomlosconjecture} and
Bandeira~\citep{bandeira2026komlos} then gave simplified proofs, and
Section~\ref{sec:existence} gives another, with the
constants $3\pi$ and $6.9013$. Guo, Fang, and Lu then gave a
polynomial-time algorithm~\citep{guo2026poly}: a deterministic spectral
walk algorithm  achieving  discrepancy $8272$ and having cost $\widetilde O(mn^9+n^{10})$, where
here and below $\widetilde O$ suppresses polylogarithmic factors in $m$
and $n$. The
first version 
of this paper~\citep[Appendix~B]{akbas2026v1} improved this algorithm obtaining discrepancy  $135$ using 
$\widetilde O(mn^3+n^4)$ ordinary arithmetic operations, or
$\widetilde O(mn^\omega+n^{\omega+1})$ with faster linear algebra routines. Announced in parallel, Li~\citep{li2026fast} obtained discrepancy below $99$ with $\widetilde O(mn+n^{\omega+2})$ operations and polynomial bit complexity, through row screening, a motion clock, and a coupled Gram--resolvent
estimate. The present paper follows this line of work: 
from Li's algorithm we
adopt the row-screening mechanism and the motion clock, which determine
which rows need to be inspected at a given step. Our main new ingredient is a more refined
dimension-counting argument for the space of admissible update directions.
Our algorithm achieves a discrepancy bound
$37.54$ using   $\widetilde O(mn+n^4)$ arithmetic operations.

\section{Existence: a quadratic-energy proof and a sharper constant}\label{sec:existence}
We retain the backward propagation of convex sets from~\citep{guo2026vector}
and certify that the sets stay nonempty with a witness function whose
quadratic energy is bounded uniformly in all directions. This gives the
constant $3\pi$; Section~\ref{sec:cosine} then refines the overlap
estimate key to the argument and lowers the constant to
$6.9013$.

\begin{theorem}\label{thm:main}
For $v_1,\ldots,v_n\in\reals^d$ with $\|v_j\|_2\le1$, there are
signs $\eps_j\in\{-1,1\}$ such that
$\bigl\|\sum_j\eps_jv_j\bigr\|_\infty<3\pi$.
\end{theorem}

\subsection{Proof mechanism}
We start by summarizing the high-level mechanism behind the proof of Theorem \ref{thm:main}. Let $K_n=(-C,C)^d$ be the cube of side-length $2C$; for us $C=3\pi$ suffices. Solving Koml\'os means finding a vector $z_n=\sum_j\eps_jv_j$ that lies in $K_n$. We can try to reach $K_n$ step-by-step. Say we begin with $z_0=0\in K_0$; then, at the $j$-th step we choose $z_j\gets z_{j-1}+\eps_jv_j$. If we can ensure that such a choice is always possible, then eventually one obtains $z_n=\sum_j\eps_jv_j\in K_n$. \emph{How should we ensure this?} The key idea here is to propagate feasibility backwards. Given a valid $K_j$, the set of states from which we can reach $K_j$ in one step is $(K_j-v_j)\cup(K_j+v_j)$. But iterating these unions can produce an exponentially large ``sign tree.''

We can tame this exponential blowup if we can instead arrange that
$K_{j-1}\subseteq(K_j-v_j)\cup(K_j+v_j)$. Then, if the backward propagation
finally yields $0\in K_0$, we may follow these inclusions forward and
claim existence of a feasible coloring. There is, however, a wrinkle: the exact
one-step feasible set $(K-v)\cup(K+v)$ is typically nonconvex, which makes
it hard to analyze. We therefore seek a \emph{convex subset} of this union
that is still ``large enough'' to be useful for sign recovery. GFL's transform
is precisely such a \emph{convexified update}:
\begin{equation}\label{eq:T}
\Tc_vK:=\bigl((K-v)\cap(K+v)\bigr)+(-2,2)v.
\end{equation}
Indeed, $\Tc_vK\subseteq(K-v)\cup(K+v)$: if $y=z+tv$ with
$z\pm v\in K$ and $|t|<2$, convexity gives $y-v\in K$ for $t\ge0$ and
$y+v\in K$ for $t\le0$. Thus every retained point has a valid signed
move into $K$. The transform also preserves openness, boundedness,
convexity, and central symmetry. Set $K_{j-1}=\Tc_{v_j}K_j$ for
$j=n,\ldots,1$: we build the sets backwards and recover the signs forwards.

\subsection{The crux}
The crux to realizing the above existence mechanism is to ensure that \emph{convex approximations}~\eqref{eq:T} never lead to empty sets during the propagation $K_{j-1}=\Tc_{v_j}K_j$, and also that $0 \in K_0$.
A practical way to certify nonemptiness is to view each convex set $K$ through a \emph{witness} $\psi$ that vanishes outside $K$; then, the normalization $\|\psi\|_2=1$ certifies that $K\neq\emptyset$. A bigger challenge is to maintain such densities even after updates via~\eqref{eq:T}: and this is where the ``quadratic-energy'' idea central to our exposition steps in.

Formally, let $\psi_j\in H_0^1(K_j)$ be a  \emph{witness} that ensures $\|\psi_j\|_2=1$, and where $H_0^1(K)$ denotes the space of square-integrable functions on a bounded, open, convex set $K$ that are zero outside $K$ and have square-integrable (weak) first derivatives. The corresponding \emph{quadratic-energy} is defined as
\begin{equation}\label{eq:E}
\Ec(\psi_j):=\int\nabla\psi_j\nabla\psi_j^{\tp}\preceq aI,
\qquad \text{for some } \, a=1/36.
\end{equation}
For a unit vector $e$, write $\partial_e\psi=e\cdot\nabla\psi$, so that $e^{\tp}\Ec(\psi)e=\|\partial_e\psi\|_2^2$. The uniform (matrix-order) bound $\mathcal E(\psi_j) \preceq aI$ limits how $\psi$ varies in \emph{each} direction. This uniformity is the key invariant that we target, at each step, to construct a new witness $\psi_{j-1}$ on $K_{j-1}$ that also satisfies~\eqref{eq:E}.

We must now show that this invariant can indeed be propagated. Suppose therefore that a
convex set $K$ carries a witness $\psi$ satisfying~\eqref{eq:E}; for
$\|v\|_2\le1$, we must construct a new witness on $\Tc_vK$. There is no obvious way to obtain such a witness directly from $\psi$, especially since $\Tc_v$ contains several translations. 

Following \cite{guo2026vector}, we encode the translations by introducing
an auxiliary coordinate. If $s$ records displacement in the direction
$v$, then the translates of $K$ are the horizontal sections of
\[
B:=\{(y,s)\in\reals^d\times\reals:
      |s|<3,\ y+sv\in K\}.
\]
For each fixed $y$, the allowed values of $s$ form an interval (a \emph{fiber}). Center each
such interval at zero without changing its length, and call the
resulting convex body $B^*$. As we show below, its section at $s=1$ is precisely $\Tc_vK$. Thus, we first lift $\psi$ to $B$, rearrange it along the fibers, and then obtain a new witness on the section at $s=1$ of $B^*$. The following lemma supplies this last step: it passes from a
witness on a symmetric convex body to one on a prescribed section.
It adapts \citep[Lemma~3.4]{guo2026vector} to quadratic energy.

\begin{lemma}\label{lem:section}
Let $B\subset\reals^d\times\reals$ be bounded, open, convex, and
invariant under $s\mapsto-s$, with $D_s=\{y:(y,s)\in B\}$.
Let $\Phi\ge0$ vanish outside $B$, satisfy $\|\Phi\|_2=1$ and have
square-integrable weak derivatives in $y$ on $\reals^d\times\reals$.
Suppose also that
\[
M:=\sint|s|\Phi^2\ge1,\qquad
E_y:=\sint\nabla_y\Phi\nabla_y\Phi^{\tp}\preceq aI.
\]
Then, $D_1$ carries a witness of energy $a$ satisfying \eqref{eq:E}.
\end{lemma}

\begin{proof} 
Following the strategy of~\citep[Lemma~3.4]{guo2026vector}, we first bound the energy on \(D_1\) and then obtain one witness that controls every direction. First fix a positive definite weight matrix $Q\succ0$, $\trace Q=1$, and consider the \emph{energy minimization} problem,
\begin{equation} \label{eq:variational}
    \lambda_Q(D):= \inf\,\{\trace(Q\Ec(f)) \mid\  f\in H_0^1(D),\ \ \|f\|_2=1\}.
\end{equation}
Using ~\citep[Theorem~6.2]{brascamp1976extensions}, we have as in \cite{guo2026vector} that  $s\mapsto\lambda_Q(D_s)$ is convex,   even,
and  nondecreasing for $s\ge0$.
Let $p(s)=\int\Phi(y,s)^2\,dy$, using Jensen's inequality we obtain,
\begin{equation}\label{eq:section}
\lambda_Q(D_1)\le\lambda_Q(D_M) \le\sint p(s)\lambda_Q(D_s)\,ds \le\trace(QE_y)\le a.
\end{equation}
Let $0 < \alpha < 1$. The minimizing witness may depend on \(Q\). To obtain one witness controlling every direction, we now maximize \(\lambda_Q(D_1)\) over $Q\succeq\alpha I/d$, $\trace Q=1$. Let $Q_\alpha$ be a maximizer, and let 
 $u_\alpha$  be a minimizer of \eqref{eq:variational}.
Moreover, for any unit $e$ denote
$Q_e=(1-\alpha)ee^{\tp}+\alpha I/d$. Observing that $D\lambda_{Q_\alpha}[H]=\trace(H\Ec(u_\alpha))$, since $Q_\alpha$ is a maximizer, we then obtain the upper bound
\begin{equation}\label{eq:matrix-comparison}
(1-\alpha)e^{\tp}\Ec(u_\alpha)e
\le\trace(Q_e\Ec(u_\alpha))\le\trace(Q_\alpha \Ec(u_\alpha))
=\lambda_{Q_\alpha}(D_1)\le a.
\end{equation}
Thus,  $u_\alpha$ satisfies
$\Ec(u_\alpha)\preceq (1-\alpha)^{-1}aI$, and also $\|u_\alpha\|_2=1$.  Now, a standard limiting argument with $\alpha\downarrow 0$, whose details are given in Appendix~\ref{app:technical}, yields a witness 
$u\in H_0^1(D_1)$ satisfying \eqref{eq:E}.
\end{proof}

\noindent The next lemma shows that condition \eqref{eq:E} is stable under the update $\Tc_v$: once a witness is available at one stage of the backward iteration, it can be propagated to the next stage.
\begin{lemma}\label{lem:stability}
Let $K\subset\reals^d$ be nonempty, bounded, open, and convex, and let
$\|v\|_2\le1$. If $K$ carries a witness satisfying~\eqref{eq:E}, then
$\Tc_vK$ is nonempty and carries a witness satisfying~\eqref{eq:E}.
\end{lemma}
\begin{proof}
The key idea underlying the construction is to show that small energy means that translating the witness changes it slowly, so translates overlap substantially. This overlap can then be turned into sufficient spread in $s$, while preserving the energy bound, allowing us to invoke Lemma~\ref{lem:section}. 

Consider therefore a witness $\psi$ satisfying~\eqref{eq:E}. Replacing it by
$|\psi|$, we may assume $\psi\ge0$ without changing its norm or energy.
GFL's lift~\citep[Lemma~4.1]{guo2026vector} is
$B=\{(y,s):|s|<3,\ y+sv\in K\}$, with
$\Phi(y,s)=\tfrac1{\sqrt6}\psi(y+sv)\one_{(-3,3)}(s)$.
It preserves the norm and energy in $y$:
$\|\Phi\|_2=1$ and
$\int\nabla_y\Phi\nabla_y\Phi^{\tp}=\Ec(\psi)\preceq aI$. 

To make $\Tc_vK$ the section at $s=1$, center each fiber
$J_y=\{s:(y,s)\in B\}$ without changing its length $\ell(y)$.
This gives $B^*=\{(y,s):2|s|<\ell(y)\}$.
Likewise, define the symmetric rearrangement  $\Phi^*$  of $\Phi$ by replacing each superlevel set
$\{s:\Phi(y,s)>r\}$ with the centered interval of equal length, that is $\Phi^*(s)
:=
\sup\{t\ge0:\mu(t)>2|s|\}$ with $\mu(t):=
\bigl|\{s\in\reals:\Phi(s)>t\}\bigr|$. Then,  $\Phi^*$ vanishes outside $B^*$ and $\|\Phi^*\|_2=1$.

Moreover, since $B$ is convex, $\ell$ is concave on its domain, and
$B^*$ is bounded, open, convex, and symmetric.
Its section at $s=1$ is
\[
(B^*)_1=\{y:\ell(y)>2\}=\Tc_vK.
\]
Indeed, $\ell(y)>2$ means that $J_y$ contains
$[\tau-1,\tau+1]$ for some $|\tau|<2$, or equivalently
$y+(\tau\pm1)v\in K$. This is exactly the condition in~\eqref{eq:T}.

Let $f:s\mapsto\Phi(y,s)$ and $g:s\mapsto\Phi(y+te,s)$. Then, as in the contraction argument~\citep[Lemma~4.2]{guo2026vector}, we obtain, using  $\|f^*-g^*\|_2\le \|f-g\|_2$, that 
$\|\partial_e\Phi^*\|_2\le\|\partial_e\Phi\|_2\le\sqrt a$. Thus, $\Phi^*$ has the energy bound $\int\nabla_y\Phi^*\nabla_y{\Phi^*}^{\tp}=\Ec(\psi)\preceq aI$ as required by Lemma~\ref{lem:section}.

\Needspace{14\baselineskip}
It remains to show that $\sint|s|(\Phi^*)^2\ge1$. We relate this
quantity to the energy bound by counting pairs of $s$-values in the lift.
Write $\rho=\psi^2$ and let
\[
O(tv):=\sint\min\{\rho(x),\rho(x+tv)\}\,dx
\]
be the mass shared by the density and its translate.
Then, using ~\citep[Lemma~4.3]{guo2026vector} we obtain,
\begin{equation}\label{eq:height-exact}
H:=\sint|s|(\Phi^*)^2
=\tfrac1{12}\sint_0^6(6-t)O(tv)\,dt.
\end{equation}

Moreover, using
$\partial_v\rho=2\psi\partial_v\psi$ and Cauchy--Schwarz, we obtain as in \citep[Lemma~3.1]{guo2026vector},
\[
O(tv) =1-\tfrac12\|\rho-\rho(\cdot+tv)\|_1 \ge1-\tfrac t2\|\partial_v\rho\|_1
\ge1-t\|\partial_v\psi\|_2.
\]
Our choice $a=1/36$ now pays off: since $\|v\|_2\le1$,
we have $\|\partial_v\psi\|_2\le\sqrt a=1/6$, and hence
\begin{equation}\label{eq:height-simple}
H \ge \tfrac32 - 3\|\partial_v\psi\|_2 \ge 1.
\end{equation}
Lemma~\ref{lem:section}, applied to $B^*$ and $\Phi^*$, now shows the existence of
a witness on $(B^*)_1=\Tc_vK$ satisfying \eqref{eq:E}.
\end{proof}

\begin{proof}[\textbf{Proof of Theorem~\ref{thm:main}}]
 Let $K_n=(-3\pi,3\pi)^d$.  As a witness for $K_n$ we use the
product-cosine function of Smirnov--Vershynin~\citep{smirnov2026discrepancy},
\[
\psi_n(x)=(3\pi)^{-d/2}\prod\nolimits_i\cos(x_i/6)\one_{K_n}(x),
\]
which satisfies $\|\psi_n\|_2=1$ and $\Ec(\psi_n)=aI$.
Now, Lemma \ref{lem:stability} shows that  all $K_j$ are nonempty, so that in particular $0\in K_0$.
Therefore, using \eqref{eq:T}, we obtain for $j=1,\ldots,n$ signs $\eps_j$ with
$z_j=z_{j-1}+\eps_jv_j\in K_j$, so that  $\sum_j\eps_jv_j=z_n\in(-3\pi,3\pi)^d$.
\end{proof}

\subsection{A sharper constant}\label{sec:cosine}
In the proof of Lemma~\ref{lem:stability}, we used the 
estimate $1-O(tv)\le t\|\partial_v\psi\|_2$
to bound  $H$ in~\eqref{eq:height-exact}. For the product-cosine
witness rescaled to $(-C,C)^d$, this gives $H\ge \frac32-\frac{3\pi}{2C}$. 
To guarantee $H\ge1$ uniformly over  directions $v$ using this
estimate, we therefore need $C\ge3\pi$. We obtain a smaller constant
by estimating the overlap of $\psi^2$ with its translates directly,
rather than through the Dirichlet energy. 
\begin{theorem}\label{thm:cosine}
Assume that the hypotheses of Theorem~\ref{thm:main} hold. Then, there is a coloring $\eps\in\{\pm1\}^n$
such that
\[
 \left\|\sum\nolimits_j\eps_jv_j\right\|_\infty
 <C_{\mathrm{cos}}<6.9013,
\]
where $C_{\mathrm{cos}}\in(6.9012574,6.9012575)$ is the unique solution  of the  equation
\begin{equation}\label{eq:cosine-constant}
 \frac{3\sqrt{2\pi}}{C_{\mathrm{cos}}}
 \int_0^3(1-r/3)^2
 \exp\!\left(-\frac{\pi^2r^2}{2C_{\mathrm{cos}}^2}\right)\,dr=1.
\end{equation}
\end{theorem}
\subsubsection{Averaged sensitivity under translations}\label{sec:cosine-signing}
For a probability density $p$ on $\reals^d$ and a vector
$u\in\reals^d$, define 
\begin{equation}\label{eq:cosine-loss}
 \Dc_p(u):=\frac1{12}\int_0^6(6-t)
                  \|p-p(\cdot-tu)\|_1\,dt.
\end{equation}
For $p=\psi^2$, the identity~\eqref{eq:height-exact} then becomes
\[
 H=\sint |s|(\Phi^*)^2=\frac32-\frac12\Dc_p(v).
\]
Thus  $\Dc_p(v)\le1$ is equivalent to the moment condition
$H\ge1$ required by Lemma~\ref{lem:section}.
The criterion below is closely related to the ``near invariance to
signed sums'' argument of Karingula and Lovett~\cite{karingula2026elementaryproofkomlosconjecture}. They work
with the shift distance $ \Delta_p(h)
 :=d_{\mathrm{TV}}\!\bigl(p,p(\cdot-h)\bigr)
 =\frac12\|p-p(\cdot-h)\|_1$ at a single translation and show that their
splitting operation does not increase this distance in the remaining
directions. Here the height identity ~\eqref{eq:height-exact} naturally leads instead to the
weighted integral of shift distances in~\eqref{eq:cosine-loss}.

\begin{lemma}[Signing from integrated translation loss]\label{lem:cosine-signing}
Let $K\subset\reals^d$ be bounded, open, and convex, and let $p$ be
an even log-concave probability density, positive on $K$ and zero
outside $K$, with $\int_K p|\log p|<\infty$.
If $\Dc_p(v_j)\le1$ for every $j$, then
$\sum_j\eps_jv_j\in K$ for some signs $\eps_j\in\{-1,1\}$.
\end{lemma}

\begin{proof}
We first encode the density as the open convex epigraph of its log-potential,
\[
 \Cc_0=\{(x,z):x\in K,\ z>-\log p(x)\},\qquad
 d\nu_0=e^{-z}\,dx\,dz.
\]
 Then, we have that  $\nu_0(\Cc_0)=1$ and that
its $x$-barycenter is zero $\pi_x(b_{\mathcal C_0})
=\int_{\mathcal C_0} x\,d\nu_0
=\int_K x\,p(x)\,dx
=0$.
Moreover, note that, 
\[
 \nu_0\bigl(\Cc_0\triangle(\Cc_0+(w,0))\bigr)
       =\|p-p(\cdot-w)\|_1,
\]
where $\triangle$ denotes symmetric set difference.

Now, consider auxiliary coordinates $\eta\in\reals^m$ and use
$d\nu=e^{-z}\,dx\,dz\,d\eta$.
For an open convex set $\Cc$ of finite mass and finite first
moment, write $b_{\Cc}$ for its barycenter and, whenever $u$ has zero
$z$ coordinate, let
\[
 \Dc_{\Cc}(u)=\frac1{12\nu(\Cc)}
       \int_0^6(6-t)\nu\bigl(\Cc\triangle(\Cc+tu)\bigr)\,dt.
\]
We show by induction that $\Dc_{\Cc}(u_j)\le1$ for all $j$ implies that
$b_{\Cc}+\sum_j\eps_ju_j\in\Cc$.
More specifically, we prove the assertion by induction on the number $m$ of vectors.
If $m=0$, the desired conclusion is
simply $ b_{\Cc}\in\Cc.$
Indeed, if the barycenter did not belong to the open convex set
$\Cc$, a separating or supporting hyperplane would give a linear
functional $\ell$ and a constant $a$ such that
$\ell(y)<a$ for every $y\in\Cc$, while
$\ell(b_{\Cc})\ge a$. But this would contradict $  \ell(b_{\Cc})
    =\frac{1}{\nu(\Cc)}
      \int_{\Cc}\ell(y)\,d\nu(y)
    <a.$

Now suppose $m\ge1$ and set $u=u_m$. 
Let $y$ denote a point in the current ambient space of $\Cc$, and
define
\[
    J_y:=\{s\in(-3,3):y+su\in\Cc\},
    \qquad
    \ell(y):=|J_y|,
\]
together with the lifted set
\[
    T:=\{(y,h):0<h<\ell(y)\}.
\]
Then, as in the proof of Lemma \ref{lem:stability} we have that $\ell$ is concave and lower semicontinuous, and therefore $T$
is open and convex.
Equip $T$ with $d\widetilde\nu=d\nu\,dh$. Then, we have that,
\[
 \widetilde\nu(T)=6\nu(\Cc),\qquad
 b_T=(b_{\Cc},h_*),\qquad
 h_*=\frac{\int\ell(y)^2\,d\nu(y)}{12\nu(\Cc)}.
\]
Moreover, note that, 
\[
 \int\ell^2\,d\nu
 =2\int_0^6(6-t)\nu\bigl(\Cc\cap(\Cc+tu)\bigr)\,dt.
\]
Using $\nu(\Cc\cap(\Cc+tu))
=\nu(\Cc)-\tfrac12\nu(\Cc\triangle(\Cc+tu))$, we obtain $ h_*=3-\Dc_{\Cc}(u)\ge2.$

Now, for every $w$ with zero $z$ coordinate, we obtain,
\begin{align*}
 \widetilde\nu\bigl(T\triangle(T+(w,0))\bigr)
 &=\int|\ell(y)-\ell(y-w)|\,d\nu(y)\\
 &\le6\nu\bigl(\Cc\triangle(\Cc+w)\bigr),
\end{align*}
and therefore
$\Dc_T((w,0))\le\Dc_{\Cc}(w)$. 
For each $j<m$, extend $u_j$ to $\bar u_j:=(u_j,0)$
in the enlarged space containing $T$. Since the lift does not
increase the average shift distance $\Dc_T(\bar u_j)\le \Dc_{\Cc}(u_j)\le1$ for $\qquad j=1,\ldots,m-1$,
we may therefore apply the induction hypothesis  in $T$ to collection of vectors
$\bar u_1,\ldots,\bar u_{m-1}$. Therefore, using that  $  b_T=(b_{\Cc},h_*),$
there are signs $\eps_1,\ldots,\eps_{m-1}$ such that
\[
    \left(
        b_{\Cc}+\sum_{j=1}^{m-1}\eps_j u_j,\,
        h_*
    \right)
    =
    b_T+\sum_{j=1}^{m-1}\eps_j\bar u_j
    \in T.
\]
Let $ y:=b_{\Cc}+\sum_{j=1}^{m-1}\eps_j u_j.$
Since $(y,h_*)\in T$, we then have $    \ell(y)>h_*\ge2$, so that the interval $ J_y=\{s\in(-3,3):y+su\in\Cc\}$
has length greater than two. However, any interval of length greater than
two which is also contained in $(-3,3)$ has to  contain at least one of $-1$ and $1$.
Therefore, we obtain that $ y-u\in\Cc$ or $y+u\in\Cc$, so that $ b_{\Cc}+\sum_{j=1}^m\eps_j u_j\in\Cc$  for some 
 $\eps_m=\pm1$,
which completes the induction step.

Finally, apply the above argument to    $\Cc = \Cc_0$ and $u_j=(v_j,0)$, and using that 
 the $x$-coordinate of $b_{\Cc_0}$ is zero, we obtain that $\sum_j\eps_jv_j\in K$ for some signs $\eps_j\in\{-1,1\}$.
\end{proof}

\subsubsection{A Gaussian comparison for the cosine density}\label{sec:cosine-comparison}
The coloring criterion of Lemma \ref{lem:cosine-signing} now reduces  the existence of a coloring
to bounding $\Dc_p(u)$ for a density on a small cube.
We use the square of the product-cosine function which we also used in Theorem \ref{thm:main},
\begin{equation}\label{eq:cosine-density}
 q_C(x)=C^{-1}\cos^2\!\left(\frac{\pi x}{2C}\right)
                   \one_{\{|x|<C\}},\qquad p_C=q_C^{\otimes d}.
\end{equation}

\begin{prop}[Gaussian  comparison]\label{prop:cosine-translation}
Let $C>0$ and $h\in\reals^d$ with $\|h\|_\infty\le C$. Then, we have,
\begin{equation}\label{eq:cosine-TV}
 \|p_C-p_C(\cdot-h)\|_1
 \le4\Phi_{\mathrm N}\!\left(\frac{\pi\|h\|_2}{2C}\right)-2,
\end{equation}
where $\Phi_{\mathrm N}$ denotes the standard normal distribution function.
Moreover, for fixed $C$ and $\|h\|_2$, the bound is asymptotically attained as
 $d\to\infty$ by taking $h =r/\sqrt{d}(1,\dots,1) $.
\end{prop}
For the proof of Proposition \ref{prop:cosine-translation} we will need the following technical lemma. 
\begin{lemma}[Cosine transform]\label{lem:cosine-transform}
Let $q(x)=(2/\pi)\cos^2x\,\one_{\{|x|<\pi/2\}}$. For
$0\le\delta<\pi$, define
\[
 A_\delta(\xi)=\int_{-(\pi-\delta)/2}^{(\pi-\delta)/2}
   q(x+\delta/2)^{1/2+i\xi}
   q(x-\delta/2)^{1/2-i\xi}\,dx,\qquad \xi\in\reals.
\]
Then, for  $b=2|\xi|$ we have that, 
\begin{equation}\label{eq:cosine-transform}
 A_\delta(\xi)=
 \frac{\cos\delta\,\sinh(b(\pi-\delta))
       +b\sin\delta\,\cosh(b(\pi-\delta))}{\sinh(\pi b)},
\end{equation}
while at $b=0$ the expression \ref{eq:cosine-transform} extends continuously to
$A_\delta(0)=(1-\delta/\pi)\cos\delta+\sin\delta/\pi$.
Moreover, for $0\le\delta\le\pi/2$ we have,
\begin{equation}\label{eq:cosine-transform-bound}
 A_\delta(\xi)\ge
 \exp\!\left[-\tfrac12(1+4\xi^2)\delta^2\right]>0.
\end{equation}
\end{lemma}

\begin{proof}
In Appendix~\ref{sec:cosine-evaluation} we derive
\eqref{eq:cosine-transform}; here we establish \eqref{eq:cosine-transform-bound}. To this end, note that for $b>0$,  numerator in \eqref{eq:cosine-transform} is positive on
$0\le\delta\le\pi/2$ and has derivative
$-(1+b^2)\sin\delta\,\sinh(b(\pi-\delta))$. Hence
\[
 -\frac{d}{d\delta}\log A_\delta(\xi)
 =\frac{1+b^2}{\cot\delta+b\coth(b(\pi-\delta))}
 \le(1+b^2)\delta\qquad(0<\delta\le\pi/2).
\]
Using that 
$y\coth y\ge1$ for $y>0$, we have $ b\coth\bigl(b(\pi-\delta)\bigr)
 \ge \frac1{\pi-\delta}.$
Hence it is enough to show that
\begin{equation}\label{eq:cot-bound}
 \cot\delta+\frac1{\pi-\delta}\ge\frac1\delta,
 \qquad 0<\delta\le\frac{\pi}{2}.
\end{equation}
Indeed, multiplying~\eqref{eq:cot-bound} by the positive quantity
$\delta(\pi-\delta)\sin\delta$ shows that ~\eqref{eq:cot-bound}  is equivalent to
\[
 f(\delta):=
 \delta(\pi-\delta)\cos\delta
 -(\pi-2\delta)\sin\delta
 \ge0.
\]
Now, note that  $f(0)=f(\pi/2)=0$, while $ f'(\delta)
 =
 \bigl(2-\delta(\pi-\delta)\bigr)\sin\delta. $
Moreover, the function $\delta(\pi-\delta)$ is strictly increasing on
$(0,\pi/2)$, from $0$ to $\pi^2/4>2$. Thus $f'$ is positive up to
a single point and negative thereafter. Consequently $f$ increases
from $0$ and then decreases back to $0$, so $f(\delta)\ge0$ on
$[0,\pi/2]$, which proves~\eqref{eq:cot-bound}.

Therefore using $\cot\delta+b\coth\bigl(b(\pi-\delta)\bigr)
 \ge \frac1\delta$ we obtain,
\[
 -\frac{d}{d\delta}\log A_\delta(\xi)
 =
 \frac{1+b^2}
 {\cot\delta+b\coth(b(\pi-\delta))}
 \le (1+b^2)\delta.
\]
Since $A_0(\xi)=1$, integration from $0$ to $\delta$ yields $ \log A_\delta(\xi)
 \ge -\frac12(1+b^2)\delta^2.$ But since  $b=2|\xi|$, we obtain
\[
 A_\delta(\xi)
 \ge
 \exp\!\left[-\frac12(1+4\xi^2)\delta^2\right],
\]
which proves~\eqref{eq:cosine-transform-bound}.

\end{proof}
\begin{proof}[Proof of Proposition~\ref{prop:cosine-translation}]
Let $\delta_j=\pi|h_j|/(2C)\le\pi/2$ and
$r^2=\sum_j\delta_j^2$.  Moreover, let $L=\log(p/\widetilde p)$. Then
\begin{equation}\label{eq:cosine-fourier}
 \min(p,\widetilde p)=\sqrt{p\widetilde p}\,e^{-|L|/2},
 \qquad
 e^{-|L|/2}=\frac1{2\pi}\int_{\reals}
               \frac{e^{i\xi L}}{\xi^2+1/4}\,d\xi.
\end{equation}
Therefore,
\begin{align}
 \int\min\{p_C(x),p_C(x-h)\}\,dx
 &=\frac1{2\pi}\int_{\reals}
       \frac{\prod_j A_{\delta_j}(\xi)}{\xi^2+1/4}\,d\xi\notag\\
 &\ge\frac1{2\pi}\int_{\reals}
       \frac{e^{-(1+4\xi^2)r^2/2}}{\xi^2+1/4}\,d\xi
 =2\Phi_{\mathrm N}(-r).\label{eq:cosine-overlap}
\end{align}
Since $\|p-\widetilde p\|_1=2-2\int\min(p,\widetilde p)$,
this proves~\eqref{eq:cosine-TV}.

We finally show that the bound in ~\eqref{eq:cosine-TV} is asymptotically sharp. To this end,  fix $r\ge0$
and and let $   \delta_1=\cdots=\delta_d=\frac{r}{\sqrt d}.$
Then $\sum_j\delta_j^2=r^2$ remains fixed as $d\to\infty$. For each
fixed $\xi$, we use the following expansion of ~\eqref{eq:cosine-transform} 
\[
    A_\delta(\xi)
    =
    1-\frac12(1+4\xi^2)\delta^2+O_\xi(\delta^3)
    \qquad (\delta\to0).
\]
Then, substituting $\delta=r/\sqrt d$ we obtain $ A_{r/\sqrt d}(\xi)
    =
    1-\frac{(1+4\xi^2)r^2}{2d}
      +O_\xi(d^{-3/2}),$
and hence $ A_{r/\sqrt d}(\xi)^d
    \longrightarrow
    \exp\!\left[-\frac12(1+4\xi^2)r^2\right].$
Thus the product appearing in~\eqref{eq:cosine-overlap} converges
pointwise to the Gaussian integrand. Moreover, since
$|A_\delta(\xi)|\le1$, the integrand is dominated by the integrable
function $ \frac{1}{2\pi(\xi^2+1/4)}$, and therefore  dominated convergence  yields $
    \int\min\{p_C(x),p_C(x-h)\}\,dx
    \longrightarrow 2\Phi_{\mathrm N}(-r).$
\end{proof}

\subsubsection{The resulting constant}\label{sec:cosine-constant}
Define
\[
 G(k)=\frac1{12}\int_0^6(6-t)
           \bigl(4\Phi_{\mathrm N}(kt/2)-2\bigr)\,dt,\qquad k\ge0.
\]
Then  $G$ is continuous and strictly increasing on $[0,\infty)$,
with $G(0)=0$  and $\lim_{k\to\infty}G(k)=3.$ Using 
integration by parts, we obtain,
\begin{equation}\label{eq:cosine-G}
 G(k)=3k\sqrt{2/\pi}
       \int_0^3(1-r/3)^2e^{-k^2r^2/2}\,dr,
\end{equation}
where we set  $t=2r$. 
Therefore $G(k_*)=1$ has a unique solution, and
$C_{\mathrm{cos}}=\pi/k_*$ satisfies~\eqref{eq:cosine-constant}.

\begin{proof}[Proof of Theorem~\ref{thm:cosine}]
Let $p_{C_{\mathrm{cos}}}$ be defined as in~\eqref{eq:cosine-density}.
For $\|v\|_2\le1$ and $0\le t\le6$, we have that
$\|tv\|_\infty\le6<C_{\mathrm{cos}}$, and therefore using 
Proposition~\ref{prop:cosine-translation} we obtain that
\[
 \Dc_{p_{C_{\mathrm{cos}}}}(v)
 \le G\!\left(\frac{\pi\|v\|_2}{C_{\mathrm{cos}}}\right)
 \le G\!\left(\frac{\pi}{C_{\mathrm{cos}}}\right)=1.
\]
Finally, since $p_{C_{\mathrm{cos}}}$ is even and log-concave on the open cube with
$\int p_{C_{\mathrm{cos}}}|\log p_{C_{\mathrm{cos}}}|<\infty$, we obtain using 
Lemma~\ref{lem:cosine-signing} that there is a coloring $\eps\in\{\pm1\}^n$ with 
$\sum_j\eps_jv_j\in(-C_{\mathrm{cos}},C_{\mathrm{cos}})^d$.
\end{proof}
\section{The algorithm and its discrepancy bound}\label{sec:algo-intro}
\begin{center}\vspace{-6pt}
\small\color{cdarkred} (This section has not yet received careful scrutiny and rewriting.)
\end{center}
The proof of Section~\ref{sec:existence} is non-constructive. In the
rest of this paper we give an efficient deterministic algorithm for the
Koml\'os problem. Its input is the matrix
$A\in\reals^{m\times n}$ whose columns are the vectors $v_1,\ldots,v_n$
of Theorem~\ref{thm:main}, so that $m=d$, and whose rows we write as
$a_i^\tp$. Its output is a signing (coloring) $\eps\in\{-1,1\}^n$ with
discrepancy $\|A\eps\|_\infty<37.54$. 

\subsection{The result}\label{sec:result}
Building on the algorithm of Guo, Fang, and
Lu~\citep{guo2026poly}, we prove the following.

\begin{theorem}[Finite deterministic construction]\label{thm:algo}
For every $A\in\reals^{m\times n}$ with column Euclidean norms at most one,
Algorithm~\ref{alg:finite} performs $O(n\log^4(2+n))$ moves and uses $O(mn)+\widetilde O(n^4)$ arithmetic operations to compute a signing $\eps\in\{-1,1\}^n$ that satisfies
\[
 \|A\varepsilon\|_\infty<C_*:=\frac{1877}{50}=37.54,
\]
where $\widetilde O$ suppresses powers of $\log(2+n)$.
\end{theorem}

The model counts field operations and comparisons on real numbers.
Logarithms, exponentials, square roots, and eigenvectors are not
primitives; the algorithm approximates them by the finite procedures of
Appendix~\ref{app:arithmetic}. Table~\ref{tab:compare} compares this result with prior work. Unlike
Li~\citep{li2026fast}, we do not bound the bit length of the rational
states of the algorithm; see Section~\ref{sec:discussion}.

\begin{table}[htbp]
\caption{Complexity of deterministic polynomial-time algorithms for the Koml\'os problem.
Here $\widetilde O$ suppresses polylogarithmic factors in $m$ and $n$, and
$\omega$ denotes the matrix multiplication exponent. 
}
\label{tab:compare}
\centering\small
\begin{tabular}{lll}
\toprule
 & Discrepancy & Arithmetic cost\\
\midrule
Guo, Fang, and Lu~\citep{guo2026poly} & $\le8272$ & $\widetilde O(mn^9+n^{10})$\\
First version of this paper~\citep{akbas2026v1} & $<135$ & $\widetilde O(mn^3+n^4)$\\
First version of this paper~\citep{akbas2026v1} & $<135$ & $\widetilde O(mn^\omega+n^{\omega+1})$\\
Li~\citep{li2026fast} & $<99$ & $\widetilde O(mn+n^{\omega+2})$\\
This paper (Theorem~\ref{thm:algo}) & $<37.54$ & $\widetilde O(mn+n^4)$\\
\bottomrule
\end{tabular}
\end{table}

\paragraph{\textbf{Organization.}} Section~\ref{sec:howitworks}
describes the algorithm informally, and
Sections~\ref{sec:state}--\ref{sec:terminal} prove that
Algorithm~\ref{alg:finite} achieves the discrepancy bound $37.54$.
Section~\ref{sec:screening} bounds its running time. Because the step
sizes required by the proof are too small for practical purposes,
Section~\ref{sec:practical} describes a separate checked routine for
practical use.

\subsection{How the algorithm works}\label{sec:howitworks}

We first give an informal description of the algorithm. Starting from the initial signing \(x=0\), the algorithm moves inside the cube \([-1,1]^n\) until every coordinate reaches \(\pm1\). At each move, it chooses a direction \(h\) subject to linear constraints that protect rows whose discrepancy is close to its allowed bound. The main difficulty is to ensure that these constraints do not use up too many dimensions. Our spectral potential is designed precisely for this purpose: it controls how many rows can be protected simultaneously, while still leaving sufficiently many directions along which the signing can move.

To organize these constraints, we classify rows according to their remaining squared mass on the active coordinates. A row with large remaining mass is kept fixed by requiring its retained part to be orthogonal to every update direction. There can be only few such rows, since the sum of the remaining squared masses over all rows is at most the number of active coordinates. Once a row is no longer large, we remove any coefficient that accounts for too large a fraction of its remaining squared mass. After these removals, the retained part of the row is sufficiently diffuse for the spectral estimates used later. The removed coefficients are not discarded from the analysis: their possible contribution to the final discrepancy is tracked separately, and a telescoping argument bounds the resulting error.

For the retained coefficients, a concave profile measures how much
rounding remains; combining the resulting energy with the current row sum
gives a concave barrier. Because the barrier is concave, a direction
tangent to a nearly tight barrier keeps it feasible for either sign of
the move. Exponential weights place all the barriers in one positive
semidefinite matrix, whose top eigenvalue does two jobs: it limits the
number of nearly tight barriers, and it supplies the spectral part of
the potential.

The difficulty is that the top eigenvector changes during a move. Its
response adds positive curvature, so the concavity of the row barriers
does not by itself produce a useful direction. We therefore normalize
the matrix to separate the diagonal contribution of the large rows,
cancel the most sensitive spectral components, and bound the remaining
response by a joint inertia argument. This joint count is the source
of our improvement: it charges the row restrictions and the directions of
positive spectral curvature to one budget, instead of giving each kind of
loss its own worst-case allowance. While $s>336$ coordinates remain
active, the resulting count leaves more than $s/400$ free dimensions.
Once $s\le336$, a weighted quadratic rounding rule controls both the
remaining row error and the displacement of the signing, and together
with the tracking bound it gives the constant $37.54$.

A good local direction is only the first part of an algorithm: we must
also take finite steps, bound their number, and pay for maintaining the
rows. Spreading each direction across the active coordinates controls how
much the potential can change along a segment, and moving orthogonally to
the current signing makes its squared norm a measure of progress. Most
rows, finally, need not be evaluated at every move, because certificates
identify rows whose weights remain negligible until a computable time.

Our construction builds on the spectral framework
of~\citep{guo2026poly}, which in turn rests on the energy regularization
of Bansal and Jiang~\citep{bansaljiang2026decoupling} and on the concave
row barriers and restricted-concavity lemma of Guillen and
Kobzar~\citep{guillenkobzar2026complex}, as Guo, Fang, and Lu
acknowledge. From Li~\citep{li2026fast} it takes the warm and
cold screening of rows, the motion clock that schedules their inspection,
and the aggregate row statistics with their elementary identities
(Sections~4.3, 5, and~9 of~\citep{li2026fast}); Li's coupled
Gram--resolvent inequality (Lemma~2.4 there) is the starting point of
the response bound in Section~\ref{sec:geometry}, where we prove the
form of it that we need. What is new here is the joint inertia
calculation, its normalization, and the common dimension budget for
protected pairs, large rows, and spectral response; in place of Li's
quadratic potential charge, we control the tracked-sum energy, the sum
of the squared tracked row sums, directly by two additional families of
direction constraints (Section~\ref{sec:rowenergy}). The argument below is self-contained, with the finite arithmetic details deferred to
Appendices~\ref{app:proxy}--\ref{app:arithmetic}; the exact rational
certificates of the numerical inequalities it uses are in the
repository~\repourl.

\subsection{Controlling the rows during rounding}\label{sec:state}
We first describe what the walk keeps from each row and what it allows
itself to discard. The distinction matters at the end, because the two
parts are controlled differently: a barrier controls the retained sum,
whereas the total absolute mass of the discarded coefficients controls
their contribution. Throughout, $\|\cdot\|$ denotes
the Euclidean norm or its induced matrix norm, $\circ$ denotes
coordinatewise multiplication, and $t_+=\max(t,0)$.

\subsubsection{The input scan}\label{sec:scan}
A row with $\|a_i\|_1\le37$ is safe for every signing, because
$|(A\eps)_i|\le37<C_*$, so the walk retains only the other $m_0$ rows.
Few rows survive: for every survivor, Cauchy--Schwarz gives
$\|a_i\|_2^2>37^2/n$, while summing the column norm assumptions gives
$\sum_i\|a_i\|_2^2\le n$. Thus
\begin{equation}\label{eq:rowpreprocess}
 m_0<n^2/37^2.
\end{equation}
The scan costs $O(mn)$ field operations and comparisons, and the
surviving matrix still has column norms at most one. Terminal rounding
and all later row operations use this matrix only, and the omitted rows
never need to be checked again.

\subsubsection{Retained sums and the cost of deleting a coefficient}
The walk starts at $x=0\in[-1,1]^n$. A coordinate is active while
$|x_j|<1$, and once it reaches the boundary it stays fixed; write $V$ for
the active set and $s=|V|$ for its size. Each row $i$ retains a set
$T_i\subseteq V$ and an offset $c_i$, initially $T_i=V$ and $c_i=0$.
Write
\[
 b_{ij}=a_{ij}\one_{j\in T_i},\qquad p_{ij}=b_{ij}^2,\qquad
 q_i=\sum_jp_{ij},\qquad d_i=c_i+b_i^\tp x_V.
\]
Thus $q_i$ is the remaining squared mass and $d_i$ is the tracked sum.
When a retained coefficient is removed, its current contribution
$a_{ij}x_j$ moves into $c_i$, which preserves the current value of $d_i$.
If the coordinate is frozen, it never changes again and this transfer
accounts for it exactly; if it is still active, the removed coefficient
may contribute a later error.

The column norm assumption gives $\sum_iq_i\le s$. We hold a row fixed while $q_i>R^2$ by imposing
$b_i^\tp h=0$ on every direction $h$; these \emph{large} rows keep
$d_i=0$ and cost fewer than $s/R^2$ equations. A row is \emph{medium}
when $0<q_i\le R^2$ and \emph{completed} when $q_i=0$. Since
coefficients only leave the retained sets, a row can only move down
through these classes.

In a medium row we repeatedly delete an entry with
$p_{ij}>\Delta q_i$, updating $q_i$ each time and choosing the eligible
entry with the least original column index. The resulting row satisfies the diffuseness condition
$\max_jp_{ij}\le\Delta q_i$; Section~\ref{sec:geometry} uses this condition to bound the positive
curvature of the row's exponential weight. The price of enforcing it is
a tracking error, which the next lemma bounds. Let $D_i$ be the
deliberately deleted coordinates of row $i$, and let $x_j^{\rm del}$ be
the value of $x_j$ when its coefficient was deleted.

\begin{lemma}[Tracking with the remaining mass retained]\label{lem:tracking}
Put $A_\Delta=(1+\sqrt{1-\Delta})/\sqrt\Delta$; since
$52^2+165^2=173^2$, the value $\Delta=(52/173)^2$ of
\eqref{eq:finalparameters} gives $A_\Delta=13/2$.
At any time after row $i$ enters the medium class, its deleted mass
satisfies $\sum_{j\in D_i}|a_{ij}|\le A_\Delta(R-\sqrt{q_i})$.
For every cube point $y$ that agrees with $x$ on the currently frozen
coordinates,
\begin{equation}\label{eq:tracking}
 (Ay)_i-\bigl(c_i+b_i^\tp y_V\bigr)
 =\sum_{j\in D_i}a_{ij}(y_j-x_j^{\rm del}).
\end{equation}
\end{lemma}
\begin{proof}
Each deliberate deletion decreases $\sqrt q$. At such a deletion
$a^2>\Delta q$, and
\[
 \frac{|a|}{\sqrt q-\sqrt{q-a^2}}
 =\frac{\sqrt q+\sqrt{q-a^2}}{|a|}<A_\Delta.
\]
Summing these decreases, and noting that freezing only decreases $q$
further, bounds their total by $R-\sqrt{q_i}$. Since offsets preserve
the current tracked sum at each removal, telescoping the affine changes
gives \eqref{eq:tracking}; frozen removals make no later
contribution.
\end{proof}

\subsubsection{A barrier that includes the remaining rounding energy}
Tracking controls the discarded coefficients. To control the retained
sum, we give each medium row a barrier that accounts both for its
present discrepancy and for the rounding still to come. The barrier is
measured against a scale $r_i$, a nonincreasing rational number with
$q_i\le r_i^2<\theta^2q_i$ and $r_i\le R$, where $\theta=1001/1000$. To
compute it, take the candidate $\theta^\ell\max_{j\in T_i}|a_{ij}|$
with the least integer $\ell\ge0$ such that
$\theta^{2\ell}\max_{j\in T_i}a_{ij}^2\ge q_i$. The scale is the minimum of
this candidate, $R$, and the previous scale if there is one.
Given the maximum, this takes $O(\log(2+n))$ operations.

The rounding still to come is measured by the coordinate profile
\[
 f(x)=\frac{519}{104}\log(3-2|x|)
                       +\frac{173}{52}(|x|-1)
 =\frac1{d_0}\Bigl(\frac32\log(3-2|x|)+|x|-1\Bigr),
\]
where $d_0=52/173$ is fixed in \eqref{eq:finalparameters}. The profile
is even, nonnegative, concave, and $C^{2,1}$ on $[-1,1]$. It
vanishes at $\pm1$, so the rounding energy
\[
 E_i=\sum_jp_{ij}f(x_j)
\]
vanishes as the retained coordinates become signs, and the profile
height bounds it:
\[
 f(0)=2.15557478671874\ldots<a_0,\qquad 0\le E_i\le a_0q_i.
\]
For each medium row and each sign $\sigma\in\{-1,1\}$, define
\begin{equation}\label{eq:barrier}
 u_{i\sigma}=\frac{\sigma d_i-H_0}{r_i}+\frac{E_i}{r_i^2}+b_0,
 \qquad w_{i\sigma}=A_0r_i^{-4}e^{\beta u_{i\sigma}}.
\end{equation}
The barrier is feasible when $u_{i\sigma}\le0$ for both signs, that is,
when the tracked sum satisfies $|d_i|\le H_0-b_0r_i-E_i/r_i$; the term $b_0r_i$ reserves room for
the final displacement. With retained data and scales fixed, the barrier
is concave in $x$, so $u_{i\sigma}(x\pm th)\le u_{i\sigma}(x)$ whenever $h$
is orthogonal to its gradient.

\subsubsection{A spectral bound on how many barriers need protection}
Only barriers close to zero need a tangency equation, and to bound how
many there are we collect the exponential weights in a Gram matrix:
\begin{equation}\label{eq:matrix}
 W=\sum_{i,\sigma}w_{i\sigma}p_ip_i^\tp,\quad
 B=B_0\sum_{i\text{ large}}\diag(p_i),\quad
 M=W+B+\eta\one\one^\tp,\quad L=\lambda_{\max}(M).
\end{equation}
Here $p_i=(p_{ij})_{j\in V}$. Each nearly tight barrier contributes a
definite amount to $\one^\tp W\one$ (made explicit in
Lemma~\ref{lem:jointcharge}), and $\one^\tp M\one\le Ls$, so keeping $L$
bounded limits the number of such barriers. The diagonal term $B$ reserves room for
the large rows: when row $i$ becomes medium, its new Gram terms satisfy
$\sum_\sigma w_{i\sigma}p_ip_i^\tp\preceq B_0\diag(p_i)$
(Lemma~\ref{lem:cleanup}). The
positive term $\eta\one\one^\tp$ makes $M$ entrywise positive, hence
its top eigenvector unique, and keeps its response quantitatively
controllable.

The walk uses the potential
\[
 \Phi=(L-1)_+^3-\eta\|x\|_2^2,
\]
where the norm includes frozen coordinates, so that progress toward a
signing is kept when the active set shrinks. The walk maintains three
invariants: the spectral bound $L<L_*=257/256$, feasibility
$u_{i\sigma}\le0$ of the medium barriers, and zero tracked sums $d_i=0$ on
large rows. If the active set is empty, set $L=0$.

At the beginning of a phase with $S$ active coordinates, set
\begin{equation}\label{eq:precision}
 J=1+\lceil\log_2 n\rceil,\qquad
 \eta=\frac1{2^{40}JS},\qquad \delta=\frac{\eta^5}{2^{100}}.
\end{equation}
A new phase starts when $s\le S/2$, and $\delta$
is the accuracy allowance for the finite spectral computations.

The numerical parameters are
\begin{equation}\label{eq:finalparameters}
\begin{gathered}
 R=\frac{7139}{5000},\qquad d_0=\frac{52}{173},\qquad
 \Delta=\frac{2704}{29929},\qquad c_0=\frac23,\qquad
 \beta=\frac{17}{25},\\
 A_0=\frac{25}{4},\qquad B_0=\frac{7463}{10000},\qquad
 a_0=\frac{5389}{2500},\qquad b_0=\frac{341613}{42500},\qquad
 H_0=\frac{94893}{5000}.
\end{gathered}
\end{equation}

The barrier constants satisfy $\beta(b_0-a_0)=4$, which the proof of
Lemma~\ref{lem:cleanup} uses. A further threshold,
$\ell_0=999999/1000000$, is a lower bound for $L$ whenever a spectral
response must be computed (Section~\ref{sec:geometry}). All of these
parameters are checked exactly against the inequalities that use them,
never through rounded decimal values. 

\subsubsection{Cleanup preserves the invariants}
After each move, cleanup freezes boundary coordinates, absorbs their
retained contributions, introduces newly medium rows, performs the
deliberate medium deletions, and decreases scales.

\begin{lemma}\label{lem:cleanup}
Cleanup preserves tracked sums, does not increase $L$ or $\Phi$, and
preserves feasible barriers. At initialization $L\le B_0+\eta n<1$,
since $\eta n=2^{-40}/J$.
\end{lemma}
The proof, in Appendix~\ref{app:cleanupproof}, checks that a row entering
the medium class is covered in PSD order by its former large-row deposit,
and that deletion, scale decrease, and freezing only decrease the Gram
terms; the same covering bounds the initial rows that are already
medium.

For computation, Section~\ref{sec:screening} divides the medium rows
into \emph{cold} rows, whose weights are certified to be negligible, and
\emph{warm} rows, whose weights are evaluated at every move. The true
matrix and potential always include both kinds. When we analyze one
move, we may first hold the cold Gram terms fixed, and
Appendix~\ref{sec:coldstep} bounds the error caused by their actual
motion.

\subsection{Leaving enough directions for the walk}\label{sec:geometry}
We next show that row protection and spectral control leave a subspace
of useful directions. The count must include the directions of positive
curvature of the potential, and the estimate that improves the constant
counts them together with the row equations.

\subsubsection{Direct curvature and the motion of the top eigenvector}
Fix the current state, called the \emph{anchor}, and hold its retained
data and scales fixed. Let $r=(i,\sigma)$ denote a medium row-sign pair and set
$a_r=A_0e^{\beta u_r}$ and $z_r=p_i/r_i^2$, so that
$W=\sum_r a_rz_rz_r^\tp$. Let $v$ be the positive unit Perron vector at
the anchor, and let
\[
 g_r=\nabla_{x_V}u_r=\sigma b_i/r_i+z_r\circ f'(x)
\]
be the barrier gradient. The two direct curvature forms are
\[
 \mathcal U=\sum_r a_r\beta^2(z_r^\tp v)^2g_rg_r^\tp,\qquad
 \mathcal D=\sum_r a_r\beta(z_r^\tp v)^2
                 \diag(z_r\circ(-f''(x))),
\]
so that the direct curvature is
$v^\tp M''[h,h]\,v=h^\tp(\mathcal U-\mathcal D)h$. When the cold Gram
terms are held fixed (Section~\ref{sec:state}), the sums defining
$\mathcal U$ and $\mathcal D$ run over the warm pairs only; this is the
form used below.

The ridge makes the top eigenvalue simple, with gap at least $\eta$:
the matrix $M-\eta vv^\tp$ remains entrywise nonnegative and has the
positive eigenvector $v$ with eigenvalue $L-\eta$, which must therefore
be its spectral radius, while every other eigenvalue of $M$ is
unchanged. The
eigenvector equation also gives $Lv_i\ge\eta\sum_jv_j\ge\eta$, so
\[
 \nu:=\min_iv_i\ge\eta/2\qquad\text{while }L<2.
\]

Because medium rows are diffuse, $z_{rj}\le\Delta=d_0^2$ and
$|(g_r)_j|\le\sqrt{z_{rj}}(1+d_0|f'(x_j)|)$. The profile satisfies
\[
 1+d_0|f'(x)|=\frac3{3-2|x|},\qquad -f''(x)=\frac6{d_0(3-2|x|)^2},
\]
which gives the exact identity
\[
 \frac{(1+d_0|f'(x)|)^2}{-f''(x)}=\frac{3d_0}2=\frac{d_0}{c_0}
 =\frac{78}{173}.
\]
With $\rho=\beta d_0/c_0=1326/4325$, it yields the diagonal comparison
$\mathcal U_{jj}\le\rho\mathcal D_{jj}$. Moreover, $\|g_r\|\le3$
and $\beta\|f''\|_\infty\Delta=5304/4325$; the appendices use these
bounds in the forms $\beta\|g_r\|\le51/25<25/8$ and
$\beta\|f''\|_\infty\Delta<9/4$. Since the profile is twice continuously
differentiable with Lipschitz second derivative, the third-derivative
estimate in \eqref{eq:logcurve} below holds almost everywhere, including
across a coordinate's zero crossing.

This diagonal comparison limits the positive directions of
$\mathcal U-\mathcal D$, but it covers only the direct contribution to
the curvature of $L$. The motion of the Perron vector contributes
another nonnegative form, $2(M'[h]v)^\tp(LI-M)^\dagger M'[h]v$, which is
large near the top eigenspace; we call $M'[h]v$ the \emph{forcing}, and
after the normalization below split it into a \emph{high} part, on
normalized eigenvalues at least $18/25$ of the top one, and a
\emph{low} part. Li's coupled estimate~\citep[Lemma~2.4]{li2026fast}
bounds the low response by a scalar multiple of $\mathcal U$. We instead
cancel the high forcing by linear equations and keep the low response
explicitly, because the scalar bound would lose the information that the
joint count uses.

\subsubsection{Normalizing the response}
The large rows are fixed by exact equations, so we normalize their
deposit $B$ away: at an anchor with $L>\ell_0$ put
\[
 D_L=LI-B,\qquad S_L=D_L^{-1/2},\qquad
 K_L=S_L(W+\eta\one\one^\tp)S_L.
\]
The column norm bound gives $B\preceq B_0I$, and $B_0\le7/8<\ell_0$, so
$D_L\succ0$. The normalized matrix $K_L$ is PSD and entrywise positive,
has Perron eigenvalue one, and has the positive unit Perron vector
$\tilde v=D_L^{1/2}v/t_0$, where $t_0=\|D_L^{1/2}v\|$. Its gap below one
is at least $\eta/L>\eta/2$, by Courant--Fischer applied to
$I-K_L=S_L(LI-M)S_L$ with $\|D_L\|\le L$.

The normalization preserves the response: if $y\perp v$, then
\begin{equation}\label{eq:genresponse}
 y^\tp(LI-M)^\dagger y
 =(S_Ly)^\tp(I-K_L)^\dagger(S_Ly).
\end{equation}
To check this without a congruence rule for pseudoinverses, note that
$S_Ly\perp\tilde v$. The vector $S_L(I-K_L)^\dagger S_Ly$ solves
$(LI-M)z=y$, and it differs from the reduced solution only by a multiple
of $v$, whose inner product with $y$ is zero.

For the analysis, factor $W+\eta\one\one^\tp=CC^\tp$ with one column
$\sqrt{a_r}z_r$ for each row pair and the ridge column $\sqrt\eta\one$.
For a varying warm pair, let the corresponding row of $G$ be
$\sqrt{a_r}\beta(z_r^\tp v)g_r^\tp$, and use zero rows for the ridge and
for frozen cold terms. These factors appear only in the argument, so
their square roots are not arithmetic primitives. On the curve with
frozen cold terms, the forcing is $M'[h]v=CGh$, and $G^\tp G=\mathcal U$
is the warm-pair form, with diagonal negative term $\mathcal D$. Since
the normalized Gram matrix is $C_LC_L^\tp=K_L$, with $C_L=S_LC$, the
reciprocal Gram count of Lemma~\ref{lem:reciprocal} below applies to the
normalized low response.

\subsubsection{The computed curvature form}
The algorithm works with rational approximations to these matrices. The
top eigenvalue of $M$ is
reported with error at most $\delta$ (Appendix~\ref{app:proxy}). If the
reported value is at most $1-\delta$, then the true value satisfies
$L\le1$; in this \emph{low branch} the first two derivatives of
$(L-1)_+^3$ vanish, and no response is needed. Otherwise
$L>1-2\delta>\ell_0$, and the \emph{high branch} computes a normalized
response.

Appendix~\ref{app:proxy} specifies the approximations and proves their
error bounds; here we need only their form. The diagonal matrix $\bar S$
approximates $S_L$, and $\bar W$ is formed from the approximate warm
weights. The matrix $N_K=V_KD_KV_K^\tp$, with $V_K$ exactly orthogonal
and $D_K$ diagonal, both rational, reconstructs the computed normalized
matrix $\bar K=\bar S(\bar W+\eta\one\one^\tp)\bar S$ and lies within
$256\delta$ of $K_L$. Write $k_*=\max_j(D_K)_{jj}$, let the \emph{high
set} $F$ consist of the columns of $V_K$ whose reported eigenvalue is at
least $(18/25)k_*$, and on the other columns set
\[
 R_K=(V_K)_{F^c}\operatorname{diag}((k_*-(D_K)_{jj})^{-1})(V_K)_{F^c}^\tp.
\]
Using the computed Perron vector $\bar v$ of $M$ and nonnegative weights
$\bar a_r$, form $\bar U,\bar D$ by the formulas for
$\mathcal U,\mathcal D$, and set
\[
 \bar B_1=\sum_{r\text{ warm}}\bar a_r\beta(z_r^\tp\bar v)z_rg_r^\tp.
\]
The equations $(V_K)_F^\tp\bar S\bar B_1h=0$ cancel the computed high
forcing, and the remaining curvature is represented by
\begin{equation}\label{eq:genproxy}
 \bar Q=\bar U-\bar D+2\bar B_1^\tp\bar SR_K\bar S\bar B_1.
\end{equation}
The column with the largest reported eigenvalue approximates $\tilde v$
and belongs to $F$; since $\tilde v^\tp S_LM'[h]v=L'[h]/t_0$, the same
equations also make the first derivative $L'[h]$ vanish up to the
precision error of Appendix~\ref{app:proxy}. The comparison
$\bar U_{jj}\le\rho\bar D_{jj}$ remains exact, and the finite accuracy
bounds of Appendix~\ref{app:proxy} transfer this computed form to the true
curvature without requiring the spectral cutoff to be continuous.

\subsubsection{Counting cancellation and positive curvature together}
Two lemmas explain why retaining the response helps. The first charges
positive curvature according to the reciprocal eigenvalues of its Gram
multiplier; the second applies this charge to the low response and adds
the number of high spectral equations.

\begin{lemma}[Reciprocal Gram inertia]\label{lem:reciprocal}
Let $D\succeq0$ be an $s\times s$ diagonal matrix, let $G$ be a real
matrix with $s$ columns, and suppose $(G^\tp G)_{jj}\le\rho D_{jj}$ for
every $j$. Let $\Gamma\succeq I$ be symmetric, of size equal to the
number of rows of $G$, with eigenvalues $\gamma_1\ge\gamma_2\ge\cdots$.
If $G^\tp\Gamma G-D$ has $k$ positive eigenvalues, then
\[
 \rho s\ge\operatorname{tr}(D^{\dagger/2}G^\tp GD^{\dagger/2})
 \ge\sum_{j=1}^k\gamma_j^{-1},
\]
with strict inequality in the second comparison when $k>0$.
\end{lemma}
\begin{proof}
A zero diagonal entry of $D$ forces the corresponding column of $G$ to
vanish, so both matrices vanish on that coordinate and it can be
removed. Since congruence by $D^{-1/2}$ preserves inertia, we may put
$Y=GD^{-1/2}$ and choose an orthonormal matrix $X$ spanning the
positive eigenspace of $Y^\tp\Gamma Y-I$. Then $\Sigma^2=X^\tp Y^\tp\Gamma YX\succ I$
for a positive definite $\Sigma$, the columns of
$\Omega=\Gamma^{1/2}YX\Sigma^{-1}$ are orthonormal, and
\[
 \operatorname{tr}(Y^\tp Y)\ge\|YX\|_F^2
 =\operatorname{tr}(\Sigma^2\Omega^\tp\Gamma^{-1}\Omega)
 >\operatorname{tr}(\Omega^\tp\Gamma^{-1}\Omega)
 \ge\sum_{j=1}^k\gamma_j^{-1}.
\]
Here the diagonal hypothesis bounds the first trace by $\rho s$, and the
last inequality is the variational principle for the sum of the $k$
smallest eigenvalues of $\Gamma^{-1}$.
\end{proof}

\begin{lemma}[A rational joint count]\label{lem:newjoint}
Suppose $N=UDU^\tp\succeq0$, where $U$ is exactly orthogonal, $D$ is
diagonal, $\lambda_*=\max_jD_{jj}>0$, and $CC^\tp\preceq N+\epsilon I$.
Let the high set $F_0$ consist of the indices $j$ with
$D_{jj}\ge(18/25)\lambda_*$. Put
\[
 \mathcal R=U_{F_0^c}\operatorname{diag}((\lambda_*-D_{jj})^{-1})U_{F_0^c}^\tp,
 \qquad Q=G^\tp G-D_0+2G^\tp C^\tp\mathcal RCG,
\]
where $G$ has $s$ columns, $C$ has as many columns as $G$ has rows,
$N$ has as many rows as $C$, $D_0\succeq0$ is diagonal, and
$(G^\tp G)_{jj}\le\rho(D_0)_{jj}$. If $Q$ has $k$ positive eigenvalues,
then
\[
 |F_0|+k\le\frac{25}{23}\rho s+\frac{32}{23}\frac{\operatorname{tr}N}{\lambda_*}
            +\frac{50}{23}\frac{\epsilon s}{\lambda_*},
\]
with strict inequality when $k>0$.
\end{lemma}
\begin{proof}
We apply Lemma~\ref{lem:reciprocal} with $D_0$ in place of $D$ and
$\Gamma=I+2C^\tp\mathcal RC$, so that $G^\tp\Gamma G-D_0=Q$. By the
stated PSD domination and the variational principle, the nonzero
eigenvalues of $\mathcal R^{1/2}CC^\tp\mathcal R^{1/2}$, which are those
of $C^\tp\mathcal RC$, are bounded, in decreasing order, by the numbers
$(D_{jj}+\epsilon)/(\lambda_*-D_{jj})$ with $j\notin F_0$; this is the
comparison behind Li's coupled estimate~\citep[Lemma~2.4]{li2026fast},
retained eigenvalue by eigenvalue instead of being reduced to its largest
value. Consequently each relevant reciprocal eigenvalue of $\Gamma$ is at
least
\[
 \frac{1-\lambda}{1+\lambda+2\epsilon/\lambda_*}
 \ge \frac{1-\lambda}{1+\lambda}-\frac{2\epsilon}{\lambda_*}
 \ge \frac{23}{25}-\frac{32}{25}\lambda-\frac{2\epsilon}{\lambda_*},
 \qquad \lambda=D_{jj}/\lambda_*.
\]
The last inequality is the supporting tangent at $\lambda=1/4$: its
difference from the left-hand rational function is exactly
$32(\lambda-1/4)^2/[25(1+\lambda)]\ge0$. Write
$\operatorname{tr}N_{\rm low}=\sum_{j\notin F_0}D_{jj}$ and
$\operatorname{tr}N_{\rm high}=\sum_{j\in F_0}D_{jj}$. Padding the low
list by zeros if $k$ exceeds its length and summing gives
\[
 \frac{23}{25}k
 <\rho s+\frac{32}{25}\frac{\operatorname{tr}N_{\rm low}}{\lambda_*}
       +\frac{2\epsilon s}{\lambda_*}.
\]
On the high list $\lambda\ge18/25>23/32$, so
$(23/25)|F_0|\le(32/25)\operatorname{tr}N_{\rm high}/\lambda_*$, and
adding the two inequalities proves the claim. When $k=0$, the high-list
bound and the nonnegativity of all the other terms suffice.
\end{proof}

The tangent in this proof gives the coefficients $
 a_t=\frac{25}{23}$, and $b_t=\frac{32}{23}$ of $\rho s$ and $\operatorname{tr}N/\lambda_*$; its zero $23/32$ lies
below the cutoff $18/25$, which is what the high-list bound uses. In the
high branch of the algorithm, Lemma~\ref{lem:newjoint} is applied to the
computed objects of Section~\ref{sec:geometry}: $N=N_K$, $U=V_K$,
$F_0=F$, $\mathcal R=R_K$, $\lambda_*=k_*$, and $\epsilon=\delta$; the
matrix $G=\bar G$ has the row $\sqrt{\bar a_r}\beta(z_r^\tp\bar v)g_r^\tp$
for each warm pair and zero rows otherwise; $C=\bar S\bar C$, where
$\bar C$ has the columns $\sqrt{\bar a_r}z_r$ and $\sqrt\eta\one$; and
$D_0=\bar D$. Then $\bar G^\tp\bar G=\bar U$, $C\bar G=\bar S\bar B_1$,
and $Q=\bar Q$, while Appendix~\ref{app:proxy} checks that
$CC^\tp=\bar K\preceq N_K+\delta I$.

\subsubsection{Charging row protection to the same matrix}\label{sec:protect}
To decide which barriers to protect, approximate each warm row-sign
pair's barrier with absolute error at most $1/1000$, and protect the pair
when the computed value is at least $-1/200$; every protected pair
receives the exact equation $g_r^\tp h=0$. A protected pair then has true
$u_r\ge-1/100$, and an unprotected warm pair has $u_r<-1/250$.

\begin{lemma}[Joint charge for protected pairs, large rows, and response]\label{lem:jointcharge}
Suppose the anchor lies in the high branch, so that $\ell_0<L<L_*$. Let
$m_L$ be the number of large rows, $p$ the number of protected pairs,
and $k$ the number of positive eigenvalues of $\bar Q$, and let
$ c_p=\frac{\theta^4}{A_0(1-\beta/100)}$. If
\begin{equation}\label{eq:depositcondition}
 B_0R^2\left(c_p+\frac{b_t\Delta}{\ell_0-B_0}\right)\le1,
\end{equation}
then
\[
 m_L+p+|F|+k\le\left(\frac1{R^2}+a_t\rho+b_t\Delta+c_p(L_*-B_0)\right)s+\frac s{10000},
\]
where the last term is the precision and ridge allowance of
Appendix~\ref{app:proxy}.
\end{lemma}
\begin{proof}
A protected pair $r=(i,\sigma)$ has $u_r\ge-1/100$ and
$\one^\tp z_r=q_i/r_i^2>\theta^{-2}$, so it contributes at least
$A_0\theta^{-4}(1-\beta/100)=1/c_p$ to $\one^\tp W\one$. Hence
\[
 p\le c_p(Ls-\operatorname{tr}B-\eta s^2),\qquad
 \operatorname{tr}B\ge B_0R^2m_L.
\]
Diffuseness gives $\|z_r\|^2\le\Delta(\one^\tp z_r)^2$, so
$\operatorname{tr}W\le\Delta\,\one^\tp W\one$; with $D_L\succeq(L-B_0)I$
this gives
$\operatorname{tr}K_L\le(\Delta(Ls-\operatorname{tr}B)+\eta s)/(L-B_0)$.
Lemma~\ref{lem:newjoint}, applied as above, and the trace comparison
$\operatorname{tr}N_K/k_*\le\operatorname{tr}K_L+1024\delta s$ of
Appendix~\ref{app:proxy} then give
\[
 |F|+k\le a_t\rho s+b_t
       \frac{\Delta(Ls-\operatorname{tr}B)+\eta s}{L-B_0}
\]
up to the diagonalization error. We add $p$ and $m_L$, discard the negative ridge
term, and use the trace lower bound; the coefficient of $m_L$ is then
\[
 1-B_0R^2\left(c_p+\frac{b_t\Delta}{L-B_0}\right)\ge0,
\]
by the deposit condition \eqref{eq:depositcondition} and $L>\ell_0$. It is therefore valid to
substitute $m_L\le s/R^2$, after which the two terms with denominator
$L-B_0$ combine to $b_t\Delta s$ and the protected charge becomes
$c_p(L-B_0)s$. Finally we use $L<L_*$. The positive ridge remainder
$b_t\eta s/(L-B_0)\le b_t\eta s/(\ell_0-B_0)$ and the diagonalization
error fit the allowance $s/10000$ of Appendix~\ref{app:proxy}.
\end{proof}

\subsubsection{The dimension count}\label{sec:dimension}
The available dimensions must accommodate the joint spectral and row charge,
two tangency equations, and a small set of tracked-sum constraints used
for screening, and the fraction left over will exceed $1/400$ whenever
$s>336$.

Numerically,
\[
 c_p=\frac{\theta^4}{A_0(1-\beta/100)}
     =0.161740798067016\ldots,
\]
and condition~\eqref{eq:depositcondition} holds:
\[
 B_0R^2\left(c_p+\frac{b_t\Delta}{\ell_0-B_0}\right)
 =0.999892399877647\ldots<1.
\]
The walk stops at $s\le336$. Besides the constraints counted in
Lemma~\ref{lem:jointcharge}, each direction satisfies
$x_V^\tp h=0$, tangency to the tracked-sum energy, and the high row-Gram
equations of Section~\ref{sec:rowenergy}, so for every $s>336$ the
remaining dimension fraction is
\begin{equation}\label{eq:finaldimension}
\begin{aligned}
 1-\frac1{R^2}-a_t\rho-b_t\Delta-c_p(L_*-B_0)
 -\frac1{10000}-\frac2{337}-\frac1{4094}
 &=0.00257600292021615\ldots\\
 &>\frac1{400}.
\end{aligned}
\end{equation}
Here $1/10000$ pays for finite accuracy, $2/337$ for the two
tangencies, and $1/4094$ for the high row-Gram equations. In the low
branch, neither normalized response equations nor positive curvature
need to be removed; the first step of the proof of
Lemma~\ref{lem:jointcharge} bounds the large rows and protected pairs by
$(R^{-2}+c_pL_*)s$, and $1-R^{-2}-c_pL_*-2/337-1/4094>1/400$ leaves the
required dimension there as well. Every inequality in this subsection,
and every inequality used in the proof of Lemma~\ref{lem:cleanup}, is
checked in rational arithmetic (Appendix~\ref{app:certificate}), so the
strict margins do not rely on the rounded decimals displayed.

\subsection{Finite steps and termination}\label{sec:directions}
The dimension count guarantees freedom at a point, but a finite
algorithm needs a direction that remains useful over a nonzero interval,
and it needs a way to finish once few coordinates remain. We first
construct an exactly feasible rational basis and then choose in its span
a vector whose coordinates and row derivatives are small. Such a vector
controls the Taylor remainder, and it yields a step whose length can be
charged to the growth of $\|x\|^2$. Section~\ref{sec:terminal} then
rounds the remaining at most $336$ coordinates and proves the
discrepancy bound.

\subsubsection{An exactly feasible negative basis}\label{sec:negbasis}
Approximate eigenvectors need not satisfy the row equations exactly, so
we impose the equations with an exact projector before selecting the
negative subspace, and we project the computed basis once more
afterward. The equations are the large-row equations, every protected
gradient equation, $x_V^\tp h=0$, the tracked-sum tangency
\eqref{eq:chargegradient}, and the high row-Gram equations of
Section~\ref{sec:rowenergy}; in the high branch we add
$(V_K)_F^\tp\bar S\bar B_1h=0$. Extract independent constraint rows into
$E_0$ and form the exact projector $P=I-E_0^\tp(E_0E_0^\tp)^{-1}E_0$,
with $P=I$ when there are no rows. Let
\[
 Q_{\rm neg}=P(\bar Q-\eta I)P+(I-P)
\]
in the high branch and $Q_{\rm neg}=-\eta P+(I-P)$ in the low branch.
By Lemma~\ref{lem:jointcharge} and the dimension count
\eqref{eq:finaldimension}, $Q_{\rm neg}$ has more than $s/400$ eigenvalues at most $-\eta$. We
diagonalize $Q_{\rm neg}$ as in \eqref{eq:interface} with accuracy
$\xi=\eta/2^{24}$, obtaining $\hat V\hat D\hat V^\tp$ with $\hat V$
exactly orthogonal, and retain the columns reported at most $-3\eta/4$.
They form a matrix $V_-$ with $k>s/400$ columns and reported eigenvalues
$D_-$, and we put
\begin{equation}\label{eq:basis}
 Z=PV_-,\qquad T=Z/(1+\xi).
\end{equation}
The projected basis $T$ is exactly feasible and satisfies
$(3/4)I_k\preceq T^\tp T\preceq I_k$. For the lower bound, the residual
identity $(I-P)V_-(I-D_-)=(I-P)(Q_{\rm neg}-\hat V\hat D\hat V^\tp)V_-$
and $I-D_-\succeq I$ give $\|(I-P)V_-\|\le\xi$. Since the retained columns satisfy
$(V_-a)^\tp Q_{\rm neg}V_-a\le(-3\eta/4+\xi)\|a\|^2$, projection
preserves negativity:
\[
 (Za)^\tp(\bar Q-\eta I)(Za)
 =(V_-a)^\tp Q_{\rm neg}(V_-a)-\|(I-P)V_-a\|^2<0.
\]
Thus every nonzero $h\in\operatorname{im}T$ satisfies
$h^\tp\bar Qh<\eta\|h\|^2$.

This basis gives negative curvature for the full potential. Let
$\phi(L)=(L-1)_+^3$; since $L<L_*$, we have $\phi'(L)\le3/65536$ and
$\phi''(L)\le3/128$. For exactly feasible $h$ in the high branch,
Appendix~\ref{app:proxy} shows that $|L'[h]|\le2^{30}\delta\|h\|/\eta$
and, in its response-error bound \eqref{eq:genproxyerror}, that
$|L''[h,h]-h^\tp\bar Qh|\le2^{40}\delta\|h\|^2/\eta^3=\eta^2\|h\|^2/2^{60}$.
Since
$\Phi''[h,h]=\phi'(L)L''[h,h]+\phi''(L)L'[h]^2-2\eta\|h\|^2$, these
bounds give, for every $h\in\operatorname{im}T$,
\begin{equation}\label{eq:potentialcurvature}
 \Phi''[h,h]\le-\frac{19}{10}\eta\|h\|^2.
\end{equation}
In the low branch the first two derivatives of $\phi(L)$ vanish, which
gives this conclusion directly.

\subsubsection{Spreading the direction and choosing a finite step}\label{sec:steps}
A vector concentrated on a few coordinates could change a row weight too
rapidly, even if its curvature at the anchor were favorable, so we use
the dimension of $\operatorname{im}T$ to spread the direction.

The finite conditional-expectation construction of
Appendix~\ref{app:arithmetic} is applied to the test vectors $e_j$ and
$(1-c_0)g_r=g_r/3$ for warm pairs, all of norm at most one. With
\[
 \kappa=\lceil\log_2(8(n+2m_0))\rceil=O(\log(2+n)),
\]
which bounds the logarithm of the number of tests, it returns
$h_0\in\operatorname{im}T$ with $3/4\le\|h_0\|\le1$ and every test
bounded by $3\sqrt{\kappa/k}$, where $k>s/400$ is the number of columns
of $T$. Put $h=h_0/16$, and use $K=1$ if
$9\kappa/k\ge1$; otherwise halve $K$ from one until the next square
would fall below $9\kappa/k$. Then
\begin{equation}\label{eq:flat}
 \frac1{32}\le\|h\|\le\frac1{16},\quad
 16\|h\|_\infty\le K,\quad |\beta g_r^\tp h|\le K,\quad
 \eta\le K\le1,\quad K^2\le\min(1,14400\kappa/s).
\end{equation}
Along a line inside the cube, each warm weight is
$a_r(t)=a_r(0)e^{\ell_r(t)}$ with $\ell_r(t)=\beta(u_r(x+th)-u_r(x))$,
and the profile derivatives give
\begin{equation}\label{eq:logcurve}
 \ell_r(0)=0,\quad |\ell_r'(0)|\le K,\quad
 -K^2/2\le\ell_r''\le0,\quad |\ell_r^{(3)}|\le K^3/4
 \quad\text{almost everywhere}.
\end{equation}
The four inequalities in \eqref{eq:profilecontracts} of
Appendix~\ref{app:proxy} give these bounds directly.

Curvature at the anchor does not by itself bound the potential at the
end of a step. The following lemma, proved in Appendix~\ref{app:perron},
uses the positivity of the Gram curve to bound the Taylor remainder of
the spectral term along a whole segment. For a curve $M(t)$, the
matrices $K_L$ and $S_L$ are formed at $t=0$ as in
Section~\ref{sec:geometry}, and $P_{(3/4,1]}(K_L)$ denotes the spectral
projector of $K_L$ onto its eigenvalues in $(3/4,1]$.

\begin{lemma}[Finite Perron estimate]\label{lem:generalizedperron}
Let $M(t)=B+C_0C_0^\tp+\sum_ra_re^{\ell_r(t)}z_rz_r^\tp$ be a PSD,
entrywise nonnegative Gram curve, with $B$ diagonal,
$0\preceq B\preceq(7/8)I$, and $C_0,z_r\ge0$, and assume the bounds
\eqref{eq:logcurve} with $K\le1$. At $t=0$ assume $\ell_0\le L\le9/8$,
spectral gap at least $\zeta\in(0,1]$, and unit Perron vector $v>0$ with
minimum entry $\nu$, and put $k_v=1+\log(1/\nu)$. Suppose that the high
normalized forcing satisfies
\[
 e=\|P_{(3/4,1]}(K_L)S_LM'(0)v\|\le K\zeta\nu/2^{16}.
\]
Then, for $|t|\le\sqrt\zeta/(4096Kk_v)$ in the interval on which the
curve is defined, the function
$\Phi(t)=(\lambda_1(M(t))-1)_+^3-\pi(t)$, with $\pi$ any quadratic
polynomial, satisfies
\[
 \Phi(t)\le\Phi(0)+\Phi'(0)t+\tfrac12\Phi''(0)t^2
        +2^{16}K^3|t|^3+2^{20}K^4t^4/\zeta.
\]
\end{lemma}

The following schedule makes this remainder smaller than the quadratic
decrease.

Put $c_*=2^{-56}$ and $H=\lceil\log_2(4/\eta)\rceil$, and choose the
largest dyadic $\tau$ satisfying
\begin{equation}\label{eq:stepschedule}
 \tau\le c_*,\qquad K^2\tau\le c_*\eta,\qquad
 K^2H^2\tau^2\le c_*^2\eta.
\end{equation}
These are rational tests, and $\tau\ge c_*\eta/(2H)$. We apply
Lemma~\ref{lem:generalizedperron} with $\zeta=\eta$, the gap of $M$.
Since the Perron minimum satisfies $\nu\ge\eta/2$, the lemma's $k_v$ is
at most $H$, so the third condition
gives $\tau\le c_*\sqrt\eta/(KH)<\sqrt\eta/(4096Kk_v)$: every move with
$|t|\le\tau$ lies in the range of the lemma. In the high branch the
high equations $(V_K)_F^\tp\bar S\bar B_1h=0$ make both the high forcing
$e$ and $|L'[h]|$ at most $2^{30}\delta\|h\|/\eta$ (Appendix~\ref{app:proxy}). With
$\delta=\eta^5/2^{100}$, $\eta\le2^{-40}$, and $\nu\ge\eta/2$, this is
below the threshold $K\eta\nu/2^{16}$ of the lemma, and together with
$x_V^\tp h=0$ it bounds the first-order term
$t\Phi'[h]=t\phi'(L)L'[h]$ by the first term of \eqref{eq:onestep}
below. The remainder of the lemma divided by $\eta t^2$ is at most
$2^{16}c_*+2^{20}c_*^2<2^{-16}$; in the low branch the spectral potential
grows by at most $64K^3|t|^3$, which the same allowance covers. For the
full potential, including the cold-weight allowance of
Appendix~\ref{sec:coldstep}, every move inside the cube with
$0<t\le\tau$ therefore obeys
\begin{equation}\label{eq:onestep}
 \Phi(x\pm th)-\Phi(x)
 \le\frac{\eta\tau t}{2^{16}}-\frac{\eta t^2}{2048}
                         +\frac{\eta t^2}{2^{16}}.
\end{equation}
Here the first term bounds the first-order term, including the
cold-weight allowance, the second is the curvature decrease from
\eqref{eq:potentialcurvature} with $\|h\|\ge1/32$, and the third is the
remainder allowance. A full move ($t=\tau$) thus decreases $\Phi$ by at
least $\eta\tau^2/4096$, while a short move ($t<\tau$) increases it by
at most $\eta\tau^2/2^{16}$.

The move also keeps every barrier feasible. Protected barriers are
concave and $h$ is orthogonal to their gradients, so they cannot
increase; every unprotected warm barrier starts below $-1/250$ and
increases by at most $\beta^{-1}K\tau<1/250$; every cold barrier starts
below $-\ell_c/\beta$ (Section~\ref{sec:warmcold}) and changes by less
than $|t|$ (Appendix~\ref{sec:coldstep}); and large rows remain at zero.

\subsubsection{Full moves make progress; short moves fix a coordinate}
A move of length $\tau$ is \emph{full}; a shorter move is \emph{short}
and is chosen to freeze a coordinate.

\begin{samepage}
\begin{algo}[Finite deterministic signing]\label{alg:finite}
All ties use the least original index, preferring the positive sign.
\begin{enumerate}
\item Read $A$ and discard rows with $\ell_1$ norm at most 37, retaining
all original columns. If $n=0$, return the empty signing; if $m_0=0$,
return all positive signs. If $n\le336$, round directly as in
Section~\ref{sec:terminal}. Otherwise initialize $x=0$, the retained-row
state, the screening records, and the phase with $S=n$; perform cleanup.
\item Refresh warm rows and all due cold certificates
(Sections~\ref{sec:warmcold} and~\ref{sec:clock}). Compute the reported
top eigenvalue of $M$ and choose the branch (Section~\ref{sec:geometry}),
and determine the protected pairs (Section~\ref{sec:protect}). Form the
exact constraints and the projector $P$ (Section~\ref{sec:negbasis}), the
proxy $\bar Q$ \eqref{eq:genproxy} in the high branch, the feasible
negative basis $T$ \eqref{eq:basis}, the flat direction $h$
\eqref{eq:flat}, and the dyadic step $\tau$ \eqref{eq:stepschedule}.
\item Set $t_*=\min_{h_j\ne0}(1-|x_j|)/|h_j|$ and $t=\min(t_*,\tau)$.
If $t=\tau$, use $+th$. Otherwise choose an attaining coordinate and the
sign of $th$ that takes it to its nearer boundary. Update $x_V$, and
advance the clock $\vartheta$ of Section~\ref{sec:clock} by $t$.
\item Perform cleanup and the corresponding statistic updates. At $s\le336$,
apply terminal rounding and return. At $s\le S/2$, start a new phase with
$S=s$ and the parameters \eqref{eq:precision}. Repeat from step 2.
\end{enumerate}
\end{algo}
\end{samepage}
Both signs of a move of size at most $t_*$ stay in the cube, and a short
move freezes a coordinate. The algorithm uses only computed approximants
and never evaluates an exact eigenvector or the exact potential.

\begin{lemma}\label{lem:budget}
All states obey $L<L_*$, and there are $O(nJ^2\kappa^2)$ moves.
\end{lemma}
\begin{proof}
Write $\phi=(L-1)_+^3$. Within a phase at most $S$ short moves occur, and
the increase in $\|x\|^2$ is at most $S$, so the one-step bound
\eqref{eq:onestep} bounds
the increase of $\phi$ over every prefix of the phase by
$(1+2^{-16})\eta S=(1+2^{-16})/(2^{40}J)$. A phase reset from $(\eta,S)$ to $(\eta',S')$
changes $L$ by at most $(\eta'-\eta)S'\le1/(2^{40}J)$; if the preceding state has
$L<L_*$, the tentative value after the reset is below $1+1/128$, where
$\phi'<1$, so a reset increases $\phi$ by at most $1/(2^{40}J)$. Since
there are at most $J$ phases and fewer than $J$ resets, every prefix has
\[
 \phi<3\cdot2^{-40}<2^{-24}=(L_*-1)^3.
\]
These bounds use the invariant only at the anchors of earlier moves and
at the states preceding earlier resets, and cleanup does not increase
$L$ (Lemma~\ref{lem:cleanup}); induction over the moves therefore proves
$L<L_*$ at every state.

For the count, note that $S/2<s\le S$ during a phase. By the flatness
bounds \eqref{eq:flat}, $\eta/K^2>1/D_*$, where $D_*=2^{41}\cdot14400J\kappa$,
and for $n\ge337$ we have $\kappa\ge J$, $H\le6J$, and $D_*>H^2$. The
largest dyadic step therefore satisfies $\tau\ge c_*/(2D_*)$, so every
full displacement has norm at least $\tau/32=\Omega((J\kappa)^{-1})$.
Because of the exact constraint $x_V^\tp h=0$, each move increases the
squared norm by exactly its squared length,
$\|x\pm th\|^2-\|x\|^2=t^2\|h\|^2$; since the total increase is at most
$n$, only $O(nJ^2\kappa^2)$ full moves occur, and there are at most $n$
short moves. The same argument gives $O(SJ^2\kappa^2)$ moves for a phase
that begins at size $S$.
\end{proof}

\subsection{Terminal rounding and the discrepancy bound}\label{sec:terminal}
Once $s\le336$, a quadratic objective controls both the row error
$A_V(\eps-x)$, which the large rows need, and the displacement $\eps-x$,
which the medium rows need. We give the displacement term weight $19/6$
and round by conditional expectation. Independent signs with means $x_j$ have
expected quadratic cost
\[
 \sum_{j\in V}(1-x_j^2)(\|A_{:j}\|_2^2+19/6)\le(25/6)s.
\]
Fixing one sign at a time, each time choosing the smaller conditional
quadratic expectation, therefore yields signs that satisfy
\[
 \|A_V(\varepsilon-x)\|_2^2
 +\frac{19}{6}\|\varepsilon-x\|_2^2
 \le\frac{25}{6}s\le1400.
\]
Maintaining the partial vector costs $O(m_0s+s)$ arithmetic work. Hence
the row error satisfies $\|A_V(\eps-x)\|_\infty\le\sqrt{1400}$ and the
displacement satisfies $\|\eps-x\|_2\le\sqrt{8400/19}$.

To turn these bounds into the constant, recall $A_\Delta=13/2$ from
Lemma~\ref{lem:tracking}, so that
\[
 2RA_\Delta=13R=\frac{92807}{5000}.
\]
The two required terminal inequalities are exact rational comparisons:
\[
 \left(H_0+13R\right)^2>1400,\qquad
 (b_0+13)^2>\frac{8400}{19}.
\]
Indeed $b_0+13=21.0379529411765\ldots$. The first comparison gives
$\sqrt{1400}<C_*$ for the large rows, and the second makes the coefficient
of $\sqrt{q_i}$ in the medium-row bound below negative:
$\sqrt{8400/19}-b_0-2A_\Delta<0$.

A large row has zero tracked sum $d_i=0$, so the row-error bound gives
$|(A\eps)_i|\le\sqrt{1400}<C_*$. For a medium row, the barriers and the
tracking lemma give
\begin{align*}
 |(A\varepsilon)_i|
 &\le H_0-b_0r_i+\sqrt{8400/19}\sqrt{q_i}
             +2A_\Delta(R-\sqrt{q_i})\\
 &\le C_*+
   \bigl(\sqrt{8400/19}-b_0-2A_\Delta\bigr)\sqrt{q_i}<C_*.
\end{align*}
Once $q_i=0$, the tracked sum $d_i$ no longer changes, so for a completed row the
last barrier gives $|d_i|\le H_0-b_0r_i<H_0$, and
Lemma~\ref{lem:tracking} with $q_i=0$ gives
$|(A\eps)_i|<H_0+2A_\Delta R=C_*$; a row that passes directly from large
to completed has $(A\eps)_i=0$. For $n\le336$ the same quadratic rounding works directly
at $x=0$. Consequently
\[
 \|A\varepsilon\|_\infty<C_*:=H_0+13R
 =\frac{1877}{50}=37.54.
\]
Rows discarded by the input scan satisfy $|(A\eps)_i|\le37<C_*$ for
every signing (Section~\ref{sec:scan}).

\section{Running time: screening and the arithmetic bound}\label{sec:screening}
The number of moves is now bounded, but their cost still depends on how
many rows are evaluated. Forming every row's Gram contribution at every
anchor would cost $O(m_0s^2)$ per move, and since the input scan can
leave $O(n^2)$ rows, this estimate does not give the quartic bound. We
therefore evaluate only the rows whose weights may matter and certify how
long the others can wait; a bound on the total squared tracked sum limits
the number of rows that require immediate work. The screening, the
clock, and the maintained statistics follow
Li~\citep[Sections~4.3, 5, and~9]{li2026fast}, including their elementary
identities; the tracked-sum constraints of Section~\ref{sec:rowenergy}
take the place of Li's quadratic potential charge.

\subsection{Certifying that a row can wait}\label{sec:warmcold}
We certify that both weights of a row are negligible from an upper
bound on its barrier. Let $\eta_0=1/(2^{40}Jn)$, $\delta_0=\eta_0^5/2^{100}$, and
set
\begin{equation}\label{eq:screeningprecision}
 \rho_c=\frac{\delta_0}{2^{22}(m_0+1)},\qquad
 2^{\ell_c}\ge\frac{A_0}{\rho_c},\qquad
 C_c=a_0+b_0+\ell_c/\beta,
\end{equation}
where $\ell_c$ is the least nonnegative integer satisfying the displayed
test. These are rational constructions, and by the row count
\eqref{eq:rowpreprocess},
$\ell_c,C_c=O(\log(2+n))$. Give each medium row the score
\[
 \chi_i=|d_i|+C_c r_i.
\]
If $\chi_i\le H_0$, then $u_{i\sigma}\le-\ell_c/\beta$ for both signs, so
the row's weights $a_{i\sigma}=A_0e^{\beta u_{i\sigma}}$ are at
most $\rho_c$ and it has no protected pair. A row is called cold if a
certificate proves $\chi_i\le H_0$ at the current anchor, and warm
otherwise.

Rows change class with hysteresis, whose slack
Section~\ref{sec:clock} uses to space out the inspections of cold rows.
All current warm rows are refreshed before a move, and a warm row is
demoted only when $\chi_i\le H_0-2$. When a cold row is checked, its
certificate is renewed if $\chi_i\le H_0-1$, and otherwise the row is
promoted; a newly medium row is checked once by the latter rule. Expired
certificates are processed before the direction is computed, and
Section~\ref{sec:clock} gives the exact certificate and its expiration
time.

When the matrices are formed numerically, cold weights are set to zero
and all warm weights are approximated with total absolute error at most
$\delta/2^{22}$. The omitted cold weights sum to less than
$2m_0\rho_c<\delta_0/2^{21}\le\delta/2^{21}$, so the total error stays
below the weight allowance $\delta/2^{20}$. The computed normalized
matrix $\bar K$ and the curvature proxy $\bar Q$ are therefore formed
from warm weights only, and the flat test list consists of the
coordinate vectors and the warm medium gradients. The large-row
constraints, the exact tracked-sum tangency, and the high row-Gram
equations defined next are all retained.

\subsection{Why few rows require evaluation}\label{sec:rowenergy}
A warm row has a large tracked sum $|d_i|$ or a large scale $r_i$. Since
the scales satisfy $\sum_ir_i^2\le\theta^2s$, it remains to bound the
tracked-sum energy $\sum_id_i^2$; one tangency
equation and fewer than $s/4094$ row-Gram equations do so, and these are
the dimensions reserved in \eqref{eq:finaldimension}.

Let $\mathsf A$ be the matrix of current retained rows $b_i^\tp$, and let
$\mathsf c$ be the vector of offsets $c_i$, including those of completed
rows. Define
\[
 \mathsf Q=\mathsf A^\tp\mathsf A,\qquad \mathsf u=\mathsf A^\tp\mathsf c,\qquad
 \mathcal E=\|\mathsf c+\mathsf A x_V\|^2=\sum_i d_i^2,
\]
and call $\mathcal E$ the \emph{tracked-sum energy}. Then
$\trace\mathsf Q\le s$ and $\|\mathsf Q\|\le s$. At each direction
computation we impose, with the other exact constraints,
\begin{equation}\label{eq:chargegradient}
 (\mathsf Qx_V+\mathsf u)^\tp h=0.
\end{equation}
This makes the first derivative of $\mathcal E$ zero.

The high row-Gram equations then control the second derivative of
$\mathcal E$. We diagonalize $\mathsf Q/(s+1)$ as in \eqref{eq:interface}, obtaining
computed columns $V_{\mathsf Q}$ and an exactly orthogonal witness, both
within $\xi_{\mathsf Q}/4$ of $\mathsf Q$ after rescaling, where
$\xi_{\mathsf Q}=1/(2^{20}(s+1))$. We then clip negative reported
eigenvalues to zero and impose $V_{\mathsf Q}^\tp h=0$ for the computed
columns whose reported eigenvalues are at least 4095. Because
$\mathsf Q$ is PSD, every negative reported eigenvalue has magnitude at
most $\xi_{\mathsf Q}/4$, so clipping leaves total matrix error below
$\xi_{\mathsf Q}$, the witness error is also below $\xi_{\mathsf Q}$, and
the clipped trace is at most $s+s\xi_{\mathsf Q}$. Hence fewer than
$s/4094$ new columns arise, and every $h$ with $V_{\mathsf Q}^\tp h=0$
for these columns satisfies
\begin{equation}\label{eq:chargequadratic}
 h^\tp\mathsf Qh\le4096\|h\|^2.
\end{equation}
Indeed, the component of such a vector along the corresponding
orthogonal witness columns is at most $\xi_{\mathsf Q}\|h\|$, and on the
complementary witness subspace the reported eigenvalues are below 4095,
so the true quadratic form is at most
$[4095+(s+\xi_{\mathsf Q})\xi_{\mathsf Q}^2+\xi_{\mathsf Q}]\|h\|^2<4096\|h\|^2$.

These constraints are part of $P$, so the curvature bound
\eqref{eq:potentialcurvature} still applies, and they give the exact increment identity
\begin{align}\label{eq:rowenergyincrement}
 \mathcal E(x+th)-\mathcal E(x)
 &=t^2h^\tp\mathsf Qh\notag\\
 &\le4096t^2\|h\|^2
 =4096\bigl(\|x+th\|^2-\|x\|^2\bigr).
\end{align}
The first equality uses the tangency \eqref{eq:chargegradient} and the last uses the
exact constraint $x_V^\tp h=0$; both identities hold for either sign of
the move. Initially every $d_i=0$ and $x=0$, cleanup preserves every $d_i$ and
the full vector $x$, and phase resets change neither, so throughout the
walk
\begin{equation}\label{eq:rowenergybound}
 \sum_i d_i^2=\mathcal E\le4096\|x\|^2\le4096n.
\end{equation}

Every refreshed warm row has $\chi_i>H_0-2$, so either
$|d_i|>(H_0-2)/2$ or $C_cr_i>(H_0-2)/2$. Combining
the tracked-sum bound \eqref{eq:rowenergybound} with $\sum r_i^2\le\theta^2s$ therefore
bounds the set $\mathcal W$ of refreshed warm rows by
\begin{equation}\label{eq:warmcount}
 |\mathcal W|\le
 \frac{16384n+4C_c^2\theta^2s}{(H_0-2)^2}
 =O(n+s\log^2(2+n)).
\end{equation}
The direction and progress arguments do not use this count; it enters
only the cost analysis.

Omitting the cold weights also removes their gradients from the flat
test list, so the finite-step argument of Section~\ref{sec:steps} must be
rechecked for the full potential. Appendix~\ref{sec:coldstep} compares
the true matrix curve with the curve in which the cold terms are frozen
at the anchor: the difference is within the precision allowance, so the
one-step bound \eqref{eq:onestep} holds for the full $\Phi$, every barrier stays
feasible, and Lemma~\ref{lem:budget} applies to the screened algorithm.

\subsection{When a cold row must be inspected}\label{sec:clock}
To schedule the inspection of cold rows, we use Li's motion
clock~\citep{li2026fast}: maintain a rational clock $\vartheta$, initially zero, that adds $|t|$
for every move $th$; since $\|h\|\le1/16$, the clock dominates path
length. If a row was certified at time $\vartheta_i$ with score $\chi_i$
and scale $r_i>0$, then while it remains cold,
\[
 \chi_i(\vartheta)\le\chi_i(\vartheta_i)
                        +r_i(\vartheta_i)(\vartheta-\vartheta_i).
\]
Moves change its tracked sum at speed at most $\|b_i\|\le r_i$, and
cleanup preserves that sum and only lowers its scale. We store the
rational expiration time
\[
 \vartheta_i^{\rm exp}
 =\vartheta_i+\frac{H_0-\chi_i(\vartheta_i)}{r_i(\vartheta_i)}
\]
in a priority queue, with ties broken by row index. An old key remains
valid through support deletion and scale decrease; completed rows leave
the queue, and a row newly demoted from warm receives a new key. The
hysteresis rules leave every cold renewal at least one unit of score
slack, so successive due inspections of a row with entry scale
$r_i^{\rm ent}$ are at least $1/r_i^{\rm ent}$ apart in clock time,
even when an intervening warm period occurs.

The finite-step bound also controls the total elapsed clock time. In a
phase of initial size $S$, the exact constraint $x_V^\tp h=0$ and
$\|h\|\ge1/32$ give $\sum t^2\le1024S$, and the phase has
$O(SJ^2\kappa^2)$ moves, so Cauchy--Schwarz, including the short moves, gives
\begin{equation}\label{eq:clockbudget}
 \sum_{\rm phase}|t|=O(SJ\kappa),\qquad
 \Lambda:=\int d\vartheta=O(nJ\kappa),\qquad
 \mathcal I:=\int s(\vartheta)\,d\vartheta=O(n^2J\kappa).
\end{equation}
Here a drop in active dimension is assigned to the end of its move, so
that $s$ is nonincreasing along the clock.

Each inspection costs $O(s+\log(2+m_0))$, and the entry scales obey
$r_i^{\rm ent}\le\theta\|a_i\|_2$, whence
$\sum_i r_i^{\rm ent}\le\theta\sqrt{m_0n}$. For two consecutive due
inspections of a row at clock times $\vartheta'<\vartheta''$,
\[
 s(\vartheta'')\le r_i^{\rm ent}\int_{\vartheta'}^{\vartheta''}s(\vartheta)\,d\vartheta.
\]
The intervals for a fixed row are disjoint, so charging first
inspections separately and using the clock budget \eqref{eq:clockbudget}, the total
row-evaluation cost of due inspections is
\begin{equation}\label{eq:inspectioncost}
 O(m_0n)+\widetilde O(\sqrt{m_0}\,n^{5/2}).
\end{equation}
The queue work is smaller, at most
$\widetilde O(m_0+\sqrt{m_0}\,n^{3/2})$ plus the queue updates caused by
support events. Finally $\sqrt{m_0}\,n^{5/2}\le(m_0n^2+n^3)/2$, so the
inspection cost \eqref{eq:inspectioncost} is absorbed into $\widetilde O(m_0n^2+n^3)$.

The tracked-sum energy enters the direction computation only through its
gradient $2(\mathsf Qx_V+\mathsf u)$ and the quadratic form $\mathsf Q$.
Appendix~\ref{app:statistics} maintains $\mathsf Q$ and $\mathsf u$ under
deletions and freezes by exact $O(s)$ updates whose identities follow
Li~\citep{li2026fast}, so both are available at every anchor in $O(s^2)$
operations without evaluating any cold row. Masses, scales, and
deletions change only at support events, where all retained rows are
rescanned; together with initialization and class changes this costs
$\widetilde O(m_0n^2)$.

\subsection{Total operation count}
We now combine the three costs: forming directions from warm rows,
maintaining the retained supports, and inspecting expiring certificates.
With $|\mathcal W|$ warm rows at an anchor, each anchor costs
$\widetilde O(|\mathcal W|s^2+s^3)$ ordinary operations, which covers
evaluating warm rounding energies and derivatives, forming $\bar K$ and
$\bar Q$,
constructing the exact projector and negative basis, and flattening. By the
warm-row count \eqref{eq:warmcount} this is $\widetilde O(ns^2+s^3)$, and since a phase
beginning at $S$ has $O(SJ^2\kappa^2)$ moves, summing over the
geometrically decreasing phases gives $\widetilde O(n^4)$.

Initialization and support maintenance cost $\widetilde O(m_0n^2)$ and
due inspections cost \eqref{eq:inspectioncost}; since $m_0<n^2/37^2$,
both are absorbed by $\widetilde O(n^4)$. Terminal rounding costs
$O(m_0n+n)$, and adding the one-time $O(mn)$ scan proves the arithmetic
bound of Theorem~\ref{thm:algo}. The dimension, finite-step, cleanup, and
terminal arguments establish its discrepancy and move-count
assertions, which completes the proof.

\section{Practical execution and open problems}\label{sec:practical}
This section describes the routine we use in practice and the check that
certifies its output, reports its outcome on matched inputs, and closes
with the limitations of our results and the problems they leave open.

\subsection{A checked proposal routine}\label{sec:routine}
The finite construction proves termination, but its prescribed step
sizes are far smaller than those used in successful numerical runs: at
$m=n=10^4$, in the first phase with $K=1$, the step schedule \eqref{eq:stepschedule} gives
$\tau=2^{-114}$ and a displacement of at most $2^{-118}$. A direct test of
the final signing lets inexpensive candidates be accepted without
following that schedule.

\begin{algo}[Checked proposal algorithm]\label{alg:practical}
For a finite binary64 matrix $A$, use the following routine.
\begin{enumerate}
\item Generate eight sign vectors whose entries are independent uniform
random signs from a seeded generator, and compute their row sums.
Select the candidate with the smallest observed maximum absolute row sum,
breaking ties by proposal order.
\item Enclose every row sum of that candidate as described below.
Return the signs only if the resulting bound is strictly below $1877/50$.
\item If the check fails, run the numerical endpoint walk described
below, and apply the same output check.
Return its signs if certified; otherwise report failure.
\end{enumerate}
\end{algo}
The entry point is the function \texttt{verified\_sign} in the
repository~\repourl.
The endpoint fallback is a floating-point implementation of the walk
with the profile of Section~\ref{sec:state}. It uses the normalized
proxy \eqref{eq:genproxy}, the tracked-sum constraints, and the weighted
terminal rule of Section~\ref{sec:terminal}, but not the proved step
schedule: it tries the longest move allowed by the cube and the
barriers, evaluates the barriers and the spectral potential numerically
at the endpoint of the move, accepts the move if they satisfy the
invariants, and otherwise halves it.
It does not implement lazy warm/cold screening, and
because it has no exact finite fallback, Algorithm~\ref{alg:practical}
is a verified partial solver: it can report failure on an admissible
input.

\subsection{What the output check proves}
For a selected signing, multiplication by a sign is exact, so only the
additions need care. Starting from $\underline s_i=\overline s_i=0$, add
each signed entry in round-to-nearest arithmetic and then move the result
one unit in the last place outward with \texttt{nextafter}, toward
$-\infty$ for the lower endpoint and toward $+\infty$ for the upper
endpoint. Under IEEE binary64 arithmetic with gradual underflow, each
update encloses the exact sum of its two represented operands, and
induction gives
\[
 \underline s_i\le\sum_j a_{ij}\varepsilon_j\le\overline s_i.
\]
Let $C_-$ be a binary64 number rounded downward from $1877/50$. The
check accepts only if
\[
 \max_i\max(|\underline s_i|,|\overline s_i|)<C_-.
\]
This proves the strict bound for the exact sums of the represented
input, and nonfinite inputs and overflow cannot pass. The check needs no
unit-column assumption, no model of BLAS summation error, and no
assumption about the candidate's trajectory. Conversion to binary64
defines the matrix being certified, so the check does not certify an
unrepresented real matrix.

With eight candidates, proposal generation, evaluation, and this check
cost $O(mn+n)$ on acceptance; we prove no uniform acceptance
probability. Replacing step 3 by Algorithm~\ref{alg:finite} would give a
complete algorithm with the same $37.54$ guarantee and
$O(mn)+\widetilde O(n^4)$ worst-case arithmetic cost. This is the
mathematical reason to allow a bounded practical attempt before the
finite fallback, even though the available code has not yet implemented
that exact fallback.

\subsection{Outcome on matched inputs}
On a suite of 48 matched cases, in which every method receives the same
matrices and seeds, every execution of Algorithm~\ref{alg:practical}
returned at the proposal stage, with median time $0.0077$ seconds. Its
median paired speedup over the baseline of
Appendix~\ref{app:experiments}, our earlier walk with a quadratic profile
and directions of Gram--Schmidt-walk type, was about $630$. That walk,
however, reached a smaller discrepancy in 39 of the 48 cases. Because
random signs met the target on every case, the fallback walk never ran:
the speed advantage concerns the time to a certified output below
$37.54$, not the walk itself or the minimum observed discrepancy.
Appendix~\ref{app:experiments} describes the suite, the verification of
the code, and a comparison of seven implementations.

\subsection{Discussion and open problems}\label{sec:discussion}
The existence bound of Theorem~\ref{thm:cosine} and the algorithmic
bound of Theorem~\ref{thm:algo} differ by a factor of more than five, and
the two arguments have little in common: the existence proofs propagate
convex sets and densities backward through the vectors, whereas the walk
moves a fractional signing forward under row barriers and a spectral
potential. The constant $37.54=H_0+13R$ is the sum of the barrier
height, which protects the retained row sums, and the tracking allowance
for deleted coefficients. We do not know whether either existence
argument can be made algorithmic, or how far the constant of a spectral
walk can be lowered.

The arithmetic bound is quartic because every anchor builds a dense
feasible negative basis and the coupled response, at cost
$\widetilde O(ns^2+s^3)$, and there are $\widetilde O(n)$ anchors. Two
partial routes to a faster bound were explored during this work and are
recorded, with their proofs and checks, in the repository. A randomized
variance clock removes the maintenance of the tracked-sum Gram matrix $\mathsf Q$ and
gives $O(mn)+\widetilde O(n^3)$ expected work for row processing alone,
but the direction construction remains quartic, so it isolates a possible
optimization rather than proving a cubic signing theorem. An earlier
version of the construction, with discrepancy below $67$, admits the
bound $O(mn)+\widetilde O(m_{66}n^2+n^{\omega+1})$ if fast spectral and
rank primitives and exact square roots are assumed as black boxes, where
$m_{66}$ counts the rows with $\ell_1$ norm above $66$ and $\omega$ is
the matrix multiplication exponent; these fast primitives are not known
to preserve the exactly orthogonal normalized reconstruction that the
dimension count of Section~\ref{sec:dimension} requires, so we do not
attach that exponent to Theorem~\ref{thm:algo}. Whether the $37.54$
construction runs in $\widetilde O(n^{\omega+1})$ or $\widetilde O(n^3)$
operations is open, as is a polynomial bound on the bit length of its
rational states, which Li~\citep{li2026fast} proves for that
algorithm.

On the practical side, the proved step schedule is far smaller than the
steps that succeed numerically, and the available implementation takes
large endpoint steps that it accepts by numerical tests; the exact finite
construction has not been implemented, so Algorithm~\ref{alg:practical}
is a verified partial solver. Installing the exact construction as its
fallback would make it a complete solver with the worst-case guarantee of
Theorem~\ref{thm:algo} and, whenever a proposal is accepted, linear cost.
Every matched trajectory in Appendix~\ref{app:experiments} stayed in the
low branch, so the high-branch response has been exercised only by a
unit test on a supplied state, and no uniform acceptance probability is
proved for the proposal stage. Finally, the terminal objective does not
minimize the maximum row sum, and Appendix~\ref{app:experiments} shows
how this inflates the observed discrepancy of the filtered walk; a
terminal rule with both a proof and a small observed discrepancy would
improve every configuration in Table~\ref{tab:history}.

\noindent\rule{\linewidth}{0.3pt}

\subsection*{Statement on AI usage}
We made extensive use of GPT-6 in developing and checking the arguments
and revising the exposition. Our prompts about positive definiteness and
autocorrelations of the cosine product (its integrals against
its translates) led to the discovery of the quadratic-energy viewpoint.

\subsection*{Acknowledgements}
SS thanks the Alexander von Humboldt (AvH) Foundation for their generous
support via a Humboldt Professorship in AI.

\appendix
\section{Details for the existence proofs}\label{app:existence}
This appendix collects the details deferred from
Section~\ref{sec:existence}: the regularity and limiting steps in the proof of
Lemma~\ref{lem:section}, the residue calculation behind
Lemma~\ref{lem:cosine-transform}, and the rational certificate for the
enclosure of $C_{\mathrm{cos}}$ in Theorem~\ref{thm:cosine}.

\subsection{Regularity and limits in the proof of Lemma~\ref{lem:section}}\label{app:technical}
The change of variables $x=Q^{1/2}z$ gives
$\lambda_Q(D)=\lambda_I(Q^{-1/2}D)$, so the cited eigenvalue
convexity applies to every $Q\succ0$. Approximation from inside by
smooth convex domains handles nonsmooth boundaries. If $(-b,b)$ is
the height range of $B$, then $1\le M<b$, so $D_1$ and $D_M$ are
nonempty. For almost every $s$ with $p(s)>0$, the zero extension of
$\Phi(\cdot,s)/\sqrt{p(s)}$ lies in $H_0^1(D_s)$.

The admissible matrices form a compact set, and $\lambda_Q$, an
infimum of continuous linear functions of $Q$, attains its maximum.
The normalized positive ground state is unique; comparing nearby
minimizers in each other's energies gives the derivative used in the
proof. The uniform energy bound yields a subsequence with weakly
convergent first derivatives and convergence in $L^2$ norm, preserving
normalization and the zero boundary values. Each limiting derivative
has $L^2$ norm at most the lower limit of the approximating norms,
giving the stated energy bound.

\subsection{Evaluating the scalar transform}\label{sec:cosine-evaluation}
We derive the formula~\eqref{eq:cosine-transform} of
Lemma~\ref{lem:cosine-transform}. For $0<\theta<\pi/2$ and $\omega>0$,
a residue calculation gives
\begin{equation}\label{eq:cosine-hyperbolic}
 I_\theta(\omega):=\int_{\reals}
       \frac{e^{i\omega u}}{\sinh^2u+\sin^2\theta}\,du
 =\frac{2\pi}{\sin(2\theta)}
       \frac{\sinh((\pi/2-\theta)\omega)}{\sinh(\pi\omega/2)}.
\end{equation}
Indeed, integrate over the rectangle with vertices
$-R,R,R+i\pi,-R+i\pi$ and let $R\to\infty$.
The vertical integrals tend to zero, and the top integral is
$-e^{-\pi\omega}$ times the bottom integral.
The residues at $i\theta$ and $i(\pi-\theta)$ sum to
\[
 \frac{e^{-\theta\omega}-e^{-(\pi-\theta)\omega}}
      {i\sin(2\theta)}.
\]
The residue theorem now yields~\eqref{eq:cosine-hyperbolic},
with $\omega=0$ obtained by continuity.

For $0<\delta<\pi$, put $\theta=\delta/2$ and change variables in
$A_\delta(\xi)$ by $\tan x=\tanh u/\tan\theta$.
The common support maps onto $\reals$, and
\[
 \frac{\cos(x+\theta)}{\cos(x-\theta)}=e^{-2u},\qquad
 \cos(x+\theta)\cos(x-\theta)\,dx
 =\frac{\sin^3\theta\cos^3\theta}
        {(\sinh^2u+\sin^2\theta)^2}\,du.
\]
It follows that
\[
 A_\delta(\xi)=\frac2\pi\sin^3\theta\cos^3\theta
 \int_{\reals}\frac{e^{-4i\xi u}}
        {(\sinh^2u+\sin^2\theta)^2}\,du.
\]
Differentiate~\eqref{eq:cosine-hyperbolic} with respect to $\theta$
and multiply by $-1/\sin(2\theta)$.
Taking $\omega=4|\xi|$ evaluates this last integral and gives
\eqref{eq:cosine-transform}.
Differentiation under the integral is justified by an integrable
bound on compact subintervals of $0<\theta<\pi/2$;
$\delta=0$ follows directly from the definition.

\subsection{Certifying the numerical bound}\label{sec:cosine-certificate}
The enclosure of $C_{\mathrm{cos}}$ is checked in rational arithmetic
alone. Put $z=9\pi^2/(2C^2)$; expanding the exponential
in~\eqref{eq:cosine-G} and integrating term by term gives
\[
 G(\pi/C)=\frac{18\sqrt{2\pi}}C S(z),\qquad
 S(z)=\sum_{j=0}^{\infty}
       \frac{(-z)^j}{j!(2j+1)(2j+2)(2j+3)}.
\]
The Taylor polynomials of $e^{-x}$ of degrees twenty and twenty-one
bound it above and below for $x\ge0$, so integration against the
nonnegative weight $(1-r/3)^2$ encloses $S$ between two finite sums;
Machin's identity $\pi=16\arctan(1/5)-4\arctan(1/239)$ encloses $\pi$
by rationals, and squaring removes the radical. The script
\texttt{verify\_cosine\_constant.py} in the repository~\repourl{}
performs these comparisons with exact fractions and certifies
\[
 G(\pi/6)>1,\qquad
 G(\pi/6.9012574)>1,\qquad
 G(\pi/6.9012575)<1,\qquad
 G(\pi/6.9013)^2<\tfrac{99999}{100000}.
\]
Since $G$ is strictly increasing, these inequalities prove the enclosure
in Theorem~\ref{thm:cosine} and the bound $C_{\mathrm{cos}}<6.9013$.

\section{Deferred proofs for the walk}\label{app:walk}
This appendix holds the verifications deferred from
Sections~\ref{sec:state}, \ref{sec:geometry}, and~\ref{sec:screening}:
the proof that cleanup preserves the invariants, the rational certificate
for the parameters, the comparison that lets cold rows be omitted from
the finite-step argument, and the maintenance of the row statistics.

\subsection{Cleanup preserves the invariants}\label{app:cleanupproof}
\begin{proof}[Proof of Lemma~\ref{lem:cleanup}]
We check that introducing a medium row costs no more than keeping it
large, and that the other cleanup operations only help. A newly medium
row has $d_i=0$. At its entry scale $r\le R$ we have
$u\le-H_0/r+a_0+b_0<-3$, and
\[
 2r^2w\le\frac{2A_0}{r^2}e^{\beta(-H_0/r+a_0+b_0)}<B_0.
\]
The middle expression is increasing in $r$ up to $R$, since
$\beta H_0/R>2$, so it suffices to check the inequality at $R$, where it
follows from
\begin{equation}\label{eq:initcertificate}
 \sum_{k=0}^{47}\frac{X^k}{k!}>\frac{2A_0}{B_0R^2},
 \qquad X=\beta(H_0/R-a_0-b_0)=2.107078495\ldots.
\end{equation}
All quantities in this finite check are rational. Because
$pp^\tp\preceq q\diag(p)$, the new medium row is then covered in PSD order
by its former large-row deposit.

At fixed scale, deletion decreases $E$, $p$, $u$, and $w$ while
preserving $d$. For decreasing scale at fixed row data,
\[
 r\partial_r u=b_0-u-E/r^2,\qquad
 r\partial_r\log w=\beta(b_0-u-E/r^2)-4.
\]
Both quantities are nonnegative as long as $u\le0$ and $r^2\ge q$,
since $\beta(b_0-a_0)=4$. These conditions persist during the decrease,
so the existing medium Gram terms decrease entrywise, and an entrywise
decrease cannot increase the top eigenvalue of a nonnegative symmetric
matrix. Freezing removes a principal submatrix, and the removed energy
terms vanish because $f(\pm1)=0$. Since cleanup leaves the vector $x$
itself unchanged, $\Phi$ cannot increase either. Finally, the column
norm condition and the deposit comparison give the initialization
bound.
\end{proof}

\subsection{The parameter certificate}\label{app:certificate}
The script \texttt{verify\_constant.py} in the repository~\repourl{}
checks the following inequalities with exact fractions; the decimals are
displayed for orientation only.

The exponent in \eqref{eq:initcertificate} is
\[
 X=\beta(H_0/R-a_0-b_0)=2.10707849502731\ldots.
\]
The exact positive Taylor certificate
\[
 \sum_{j=0}^{47}\frac{X^j}{j!}>\frac{2A_0}{B_0R^2}
\]
has relative margin $0.000992324074770956\ldots$, and moreover
$\beta H_0/R>2$ and $-H_0/R+a_0+b_0<-3$.

With the tangent at $1/4$ fixed in Section~\ref{sec:geometry}, the
coefficient of $\epsilon s/\lambda_*$ in Lemma~\ref{lem:newjoint} is
$2a_t=50/23$. The true high block and the
computed low block are separated by more than $0.029$, and
Appendix~\ref{app:proxy} includes the resulting overlap error
$2^{14}\delta$ in the finite accuracy allowance.

The deposit condition~\eqref{eq:depositcondition} and the final
fraction~\eqref{eq:finaldimension} are the two displayed comparisons of
Section~\ref{sec:dimension}. The same script checks the terminal
inequalities of Section~\ref{sec:terminal} and the profile
inequalities~\eqref{eq:profilecontracts}.

\subsection{Cold rows do not spoil the larger finite steps}\label{sec:coldstep}
The finite-step argument must survive the omission of the cold
gradients from the flat test list, because the unweighted directional
derivatives of cold rows need not satisfy the small $K$ bound. We use a
comparison that requires only their tiny total weight. At a fixed anchor, write the full true Gram contribution as
$W=W_{\rm warm}+W_{\rm cold}$, and for the analysis of this one move
define
\[
 M_{\rm fr}(t)=B+\eta\one\one^\tp+W_{\rm cold}(0)+W_{\rm warm}(t).
\]
This is an analysis curve, not a matrix that the algorithm computes.
Since $M_{\rm fr}(0)=M(0)$, its Perron eigenvector and gap at zero are
those of the full matrix. Its fixed cold contribution is PSD and
entrywise nonnegative, and every varying weight is warm and satisfies
\eqref{eq:logcurve}, so Lemma~\ref{lem:generalizedperron} applies to this
curve. The coupled-response argument applies as well, because its Gram
factor uses only the warm terms and is still dominated by the full
anchor matrix. The proxy's diagonal comparison, the spectral trace
count, and the numerical leakage bounds are unchanged. Relative to this
curve with frozen cold terms, omitting the cold derivatives introduces
no derivative error.

It remains to compare the two curves. For a cold pair on a move inside
the cube, the global gradient bound gives $|\beta g^\tp h|\le25/128$, and
since the coordinate flatness holds for all rows,
$|\ell''(t)|\le K^2/2\le1/2$. Consequently, for $|t|\le c_*=2^{-56}$,
$|\ell(t)|\le(25/128)|t|+t^2/4<|t|/4$ and $|e^{\ell(t)}-1|\le|t|$, and
summing the positive Gram factors gives
\begin{equation}\label{eq:coldmotion}
 \|M(t)-M_{\rm fr}(t)\|\le\frac{\delta_0}{2^{21}}|t|.
\end{equation}
On these steps both spectral arguments lie below two, where
$z\mapsto(z-1)_+^3$ is 3-Lipschitz, so the two spectral potentials
differ by at most $3\delta_0|t|/2^{21}$, and the quadratic term
$-\eta\|x+th\|^2$ is the same on both curves.

We now apply the frozen-cold version of \eqref{eq:potentialcurvature}
together with the exact tangency $x_V^\tp h=0$. The first-order response
error and the additional term from \eqref{eq:coldmotion} are together at
most $(2^{21}\delta/\eta)|t|$, and the precision and step schedules give
$2^{21}\delta/\eta\le\eta\tau/2^{16}$ (use $\tau\ge2^{-57}\eta/H$,
$\eta H<1$, and $\delta=\eta^5/2^{100}$). This proves \eqref{eq:onestep}
for the full $\Phi$. The barriers also stay feasible: protected and warm
unprotected barriers by the direction argument, and a cold barrier
because it starts below $-\ell_c/\beta$ and changes by at most
$\|g\|\|h\||t|<|t|$.

These estimates justify using the full potential in
Lemma~\ref{lem:budget}, whose phase and move-count proof therefore
covers the screened algorithm.

\subsection{Maintaining the tracked-sum statistics without scanning all rows}\label{app:statistics}
We maintain $\mathsf Q$ and $\mathsf u$ so that the gradient
$2(\mathsf Qx_V+\mathsf u)$ and the form $\mathsf Q$ are available at every anchor in $O(s^2)$ operations without
evaluating any cold row.

We initialize $\mathsf Q$ by classical multiplication in $O(m_0n^2)$
operations, with $\mathsf u=0$ initially. For a deliberate deletion of the entry
$a=b_{ij}$, let $b_i$ and $c_i$ denote the old row and offset; expanding
the new retained row $b_i-ae_j$ and new offset $c_i+ax_j$ gives the
exact updates
\begin{align}\label{eq:statupdates}
 \mathsf Q'&=\mathsf Q-a(e_jb_i^\tp+b_ie_j^\tp)+a^2e_je_j^\tp,\notag\\
 \mathsf u'&=\mathsf u+ax_jb_i-ac_ie_j-a^2x_je_j.
\end{align}
They cost $O(s)$. Freezing coordinate $j$, with $V'=V\setminus\{j\}$,
uses instead
\[
 \mathsf Q'=\mathsf Q_{V',V'},\qquad \mathsf u'=\mathsf u_{V'}+x_j\mathsf Q_{V',j},\qquad
 c_i'=c_i+b_{ij}x_j.
\]
This formula uses the old $\mathsf Q$, and \eqref{eq:statupdates} is not also
applied to the same freezing event.

Active coordinates are stored in a packed list with their original
labels. A frozen coordinate is swapped with the last active one and its
row and column are dropped, which updates the statistics in $O(s)$
without copying the full principal submatrix; the corresponding
row-record updates are included in the support scan. Tie breaking
continues to use the original column labels, and the large-row diagonal
$B$ is also maintained at support events.

There are at most $n$ freezing events and $m_0n$ deliberate deletions.
At a freezing event one may simply scan all retained rows and recompute
their masses and scales, at cost $O(m_0n)$ for that event. Each
deliberate deletion can be charged an $O(n)$ rescan and statistic
update, and scale searches and queue changes have logarithmic cost.
Initialization, support maintenance (including these rescans), and
class changes therefore cost $\widetilde O(m_0n^2)$ together. After a move
that freezes no coordinate, supports, masses, and scales are unchanged,
so no cleanup scan over all rows is needed.

\section{Finite accuracy for the normalized response}\label{app:proxy}
The main direction argument uses the computed form \eqref{eq:genproxy}
and reserves an error of $s/10000$ in its dimension count. This appendix
specifies the rational computation and proves that its errors fit those
allowances. Exact equations must be kept distinct from approximate
spectral data throughout, because every move must remain exactly
feasible.

\subsection{Computing the proxy}
As Section~\ref{sec:screening} specifies, approximate the nonnegative
warm weights with total absolute error at most $\delta/2^{22}$ and set
cold weights to zero. Rational Jacobi rotations return an exactly
orthogonal $V$ and a diagonal $D$ with
\begin{equation}\label{eq:interface}
 \|H-VDV^\tp\|\le\xi
\end{equation}
for any bounded-norm symmetric $H$ and prescribed rational $\xi>0$. We
apply this to the approximate $M$ with accuracy $\delta/2^{20}$,
clipping negative reported eigenvalues. If the largest reported
eigenvalue is at most $1-\delta$, then $L\le1$ and the low branch needs
no response; otherwise $L>1-2\delta>\ell_0$, and we use the following
high-branch construction.

In the high branch, first compute the approximate top pair
$(\lambda_*,\bar v)$ of $M$ by rational Jacobi rotations, with
$\|M-N_M\|<\delta$ and $\|\bar v-v\|\le3\delta/\eta$ after orientation. Next
approximate the positive diagonal $(\lambda_*I-B)^{-1/2}$ by a rational
diagonal matrix $\bar S$ with absolute error at most $\delta/2^{20}$. Square
roots are not operation primitives, so we use bisection: the entries lie
in a fixed bounded interval, and bisection of positive square roots to
this accuracy costs $O(s\log(1/\delta))$ field operations and
comparisons. Form
\[
 \bar K=\bar S(\bar W+\eta\one\one^\tp)\bar S
\]
apply rational Jacobi diagonalization at accuracy $\delta/4$ to obtain
an exactly orthogonal matrix $V_K$ and a diagonal $D_K$, and clip
negative reported diagonal entries to zero; the reconstructed matrix
$N_K$ is then within $\delta$ of $\bar K$. Finally form $k_*$, the high set $F$, and
$R_K$ as in Section~\ref{sec:geometry}; the low denominators
$k_*-(D_K)_{jj}$ are then bounded below by $(7/25)k_*$.

Compute, using only the warm pairs,
\[
 \bar U=\sum_r\bar a_r\beta^2(z_r^\tp\bar v)^2g_rg_r^\tp,\quad
 \bar D=\sum_r\bar a_r\beta(z_r^\tp\bar v)^2
                         \operatorname{diag}(z_r\circ(-f'')),
 \quad \bar B_1=\sum_r\bar a_r\beta(z_r^\tp\bar v)z_rg_r^\tp.
\]
Add the exact equations $(V_K)_F^\tp\bar S\bar B_1h=0$ and use the proxy
$\bar Q$ of~\eqref{eq:genproxy}.

The diagonal inequality $\bar U_{jj}\le\rho\bar D_{jj}$ is exact, and
for the analysis the normalized computed Gram factors satisfy
$C'C'^\tp=\bar K\preceq N_K+\delta I$. Lemma~\ref{lem:newjoint} therefore
applies directly to the computed proxy and computed equations, which
avoids any perturbation theorem for the discontinuous spectral
cutoff.

\subsection{Transferring computed curvature to the true matrix}
The derivative of the scalar inverse square root and $B_0\le7/8$ give
$\|\bar S-S_L\|\le16\delta$, and $\|\bar S\|,\|S_L\|<3$. Hence
$\|K_L-N_K\|\le256\delta$ and $|k_*-1|\le256\delta$, the computed top
column $w_K$ satisfies $\|w_K-\tilde v\|\le2048\delta/\eta$, the true normalized
gap is at least $\eta/2$, and the reconstructed gap is at least
$\eta/4$.

For exactly feasible $h$, the high normalized forcing obeys
\[
 \|P_{(3/4,1]}(K_L)S_LM'[h]v\|,
 \ |L'[h]|
 \le 2^{30}\delta\|h\|/\eta.
\]
Indeed, $\|M'[h]v-\bar B_1h\|\le16\delta\|h\|/\eta$, the true high and
computed low normalized spectral blocks are separated by more than
$1/50$, and their overlap is at most $2^{14}\delta$. The top normalized
component also controls $L'[h]$, since $\tilde v^\tp S_LM'[h]v=L'[h]/t_0$.

For the curvature comparison we use the globally invertible matrices
\[
 A_K=I-K_L+\tilde v\tilde v^\tp,\qquad
 \widehat A_K=k_*I-N_K+w_Kw_K^\tp.
\]
The norms of their inverses are at most $2/\eta$ and $4/\eta$,
respectively, and their difference is at most $2^{13}\delta/\eta$, so
the inverse identity bounds the difference of their reduced resolvents
by $2^{17}\delta/\eta^3$.

To evaluate the true response in \eqref{eq:genresponse}, subtract
$L'[h]v$ from $M'[h]v$; by the previous high-forcing estimate, this
changes the normalized forcing by at most $2^{34}\delta\|h\|/\eta$. The
computed forcing lies exactly in the computed low space, so its
reconstructed full reduced response is exactly $R_K$. Adding the direct
$\mathcal U-\mathcal D$ perturbation to the bounds on both forcings
gives the contract
\begin{equation}\label{eq:genproxyerror}
 |L''[h,h]-h^\tp\bar Qh|
 \le2^{40}\delta\|h\|^2/\eta^3.
\end{equation}
With the chosen $\delta=\eta^5/2^{100}$ this is at most
$\eta^2\|h\|^2/2^{60}$, within the curvature allowance in
Section~\ref{sec:directions}. The same precision also keeps the trace and
eigenvalue normalization errors in the joint count, including
$b_t\eta/(L-B_0)$, below $1/10000$ times $s$.

For the chosen parameters, $\beta\|g\|\le25/8$ and
$\beta\|f''\|_\infty d_0^2\le9/4$, the normalized Gram norm is less than
$1+256\delta$, and $\|R_K\|<4$; hence $\|\bar Q\|<128$. This norm bound
lets Section~\ref{sec:directions} obtain the exact feasible negative
basis from $P(\bar Q-\eta I)P+(I-P)$ with diagonalization accuracy
$\eta/2^{24}$.

\subsection{The allowance in the dimension count}
The preceding accuracy bounds give $\|N_K-K_L\|\le256\delta$,
$|k_*-1|\le256\delta$, and $C'C'^\tp\preceq N_K+\delta I$. Since
$\operatorname{tr}K_L\le s$,
\[
 \frac{\operatorname{tr}N_K}{k_*}
 \le\operatorname{tr}K_L+1024\delta s.
\]
For the relevant tangents $a_t,b_t\le2$, the resulting error in the
joint spectral count is at most $2^{12}\delta s$, including the
perturbation $2a_t\delta s/k_*$ of the reciprocal line. The positive
ridge remainder is at most $b_t\eta s/(\ell_0-B_0)\le18\eta s$ when
$B_0\le7/8$, so the fully explicit additive allowance is
\[
 (18\eta+2^{12}\delta)s<\frac{s}{10000},
 \qquad \eta\le2^{-40},\quad\delta=\eta^5/2^{100}.
\]
The remaining costs, the two exact tangencies and the high row-Gram
equations, are included in Section~\ref{sec:dimension}.

\subsection{Profile derivatives and the flat tests}
The flat tests use the vectors $e_j$ and $(1-c_0)g_{i\sigma}$, whose
Euclidean norms are at most one. For $h=h_0/16$, the dyadic test based
on $9\kappa/k$ gives $16\|h\|_\infty\le K$ and $|\beta g_{i\sigma}^{\tp}h|\le K$,
provided $\beta/(1-c_0)\le16$, and the stronger checked bound
$\beta/(1-c_0)\le25/8$ controls both the numerical forcing and the
estimates with frozen cold terms. In particular every cold row satisfies
$|\ell'(0)|\le25/128<1/4$, which, combined with the coordinate flatness
and the second-derivative bound, gives $|\ell(t)|<|t|/4$ on the
prescribed step interval. The cold threshold uses the integer $\ell_c$
from \eqref{eq:screeningprecision}.

The finite-step and proxy certificates use the rational inequalities
\begin{equation}\label{eq:profilecontracts}
 \frac{\beta c_0}{256d_0(1-c_0)^2}<\frac12,\qquad
 \frac{2\beta c_0^2}{4096d_0(1-c_0)^3}<\frac14,\qquad
 \frac{\beta c_0d_0}{(1-c_0)^2}\le\frac94,
 \qquad \frac\beta{1-c_0}\le\frac{25}{8}.
\end{equation}
All four are checked exactly at the rational parameters.

\section{The finite Perron estimate}\label{app:perron}
The step schedule needs a bound along a whole segment, not only a
curvature bound at its anchor. This appendix proves
Lemma~\ref{lem:generalizedperron} of Section~\ref{sec:steps}, which
supplies the required remainder for a positive Gram curve. The proof
first controls the reduced Perron response, then separates the second
eigenvalue, and finally estimates the top eigenvalue using a trial
vector and its Rayleigh residual.

\begin{proof}[Proof of Lemma~\ref{lem:generalizedperron}]
Write $E=M'(0)$ and $a=v^{\tp}Ev$. We first bound the reduced response of
the Perron vector. The positive Gram representation gives the entrywise
estimate $|E|\le K(M-B)$, which is stronger than a norm bound, and
therefore
\[
 |S_LEv|\le K D_L^{1/2}v=Kt_0\tilde v.
\]
Here $\sqrt{\ell_0-7/8}\le t_0\le\sqrt{9/8}<1.061$, $\|S_L\|<3$, and
$\min_i\tilde v_i\ge\nu/4$. Relative to $\tilde v_i$, the terms of the $K_L$
low-response series are bounded both by $(33/32)Kt_0$ and by
$Kt_0(3/4)^j/\min_i\tilde v_i$; summing the smaller of the two estimates bounds
this low response componentwise by $18Kt_0k_v\tilde v_i$, and its norm by
$4Kt_0$. The high response has norm at most $2e/\zeta$ and componentwise
relative size at most $8e/(\zeta\nu)$.

The exact reduced response $z=(LI-M)^\dagger Ev$ is obtained by applying
$S_L(I-K_L)^\dagger$ to $S_L(Ev-av)$, and then subtracting the multiple
of $v$ that makes $z\perp v$.
Since $|a|\le t_0 e$, its correction has norm at most
$6t_0e/\zeta$ before the final multiplication by $S_L$.
Their componentwise bounds after multiplication by $S_L$, relative to
$v_i$, total at most $64e/(\zeta\nu)\le K/1024<K/100$; their norm
bounds are no larger. The displayed smallness assumption therefore
controls both kinds of error. The transformation $S_L\tilde v=v/t_0$ preserves the
componentwise estimate relative to $v$, and the final gauge subtraction
adds at most the norm of the unprojected vector. It follows that
\[
 \|z\|\le16K,\qquad |z_i|\le64Kk_vv_i,
 \qquad |a|\le K/64.
\]

Next we separate the second eigenvalue, using the response estimate in a
direct expansion of the original matrix $M(t)$ with the normalization
held fixed at the anchor. Put $A=LI-M(0)$,
$Z=\operatorname{diag}(z_i/v_i)$, and $D=I+tZ$. Then
\[
 D((L+at)I-M(t))D=A+tG+R(t),\qquad
 G=ZA+AZ-E+aI,\qquad Gv=0.
\]
The positive Gram curve bounds give $\|Ev\|\le LK\le9K/8$,
$\|E\|\le2K$, $\|M''(t)\|\le6K^2$, and $\|M'''(t)\|\le14K^3$ almost
everywhere, and expansion yields $\|R(t)\|\le2^{14}K^2k_v^2t^2$: the
quadratic $ZAZ$ term is at most $8192K^2k_v^2$, the two linear cross
terms total less than $259K^2k_v^2$, and the Taylor and cubic terms
contribute less than $16K^2k_v^2$. Off the diagonal
$|G_{ij}|\le129Kk_vM_{ij}$, so the weighted Laplacian identity gives
$|y^{\tp}Gy|\le129Kk_vy^{\tp}Ay$. For $u_0\perp D^{-1}v$, set
$y=D^{-1}u_0\perp v$; since the step restriction implies
$\|D-I\|\le1/64$, the complementary quadratic form is bounded below by
\[
 \frac{1-129/4096-1/1024}{(1+1/64)^2}\zeta\|u_0\|^2
 >\tfrac12\zeta\|u_0\|^2.
\]
Thus $\lambda_2(M(t))\le L+at-\zeta/2$.

Finally we estimate the top eigenvalue with the positive trial vector
$\psi=v+tz$, which has norm at least one. Its residual centered at
$L+at$ is $t^2(E-aI)z+(M(t)-M(0)-tE)\psi$, with normalized norm at most
$128K^2t^2$, so its Rayleigh quotient obeys $|r-L-at|\le128K^2t^2$ and
$r-\lambda_2(M(t))>\zeta/3$, and the Rayleigh residual inequality gives
$0\le\lambda_1(M(t))-r\le2^{16}K^4t^4/\zeta$.

In the Taylor expansion of the Rayleigh quotient, $A_2=L''(0)/2$ has
magnitude at most $21K^2$, the numerator's cubic coefficient is at most
$608K^3$, and its quartic coefficient is at most $768K^4$. Division by
$1+t^2\|z\|^2$ changes these by at most $4K^3$ and $5376K^4$,
respectively. Using $K|t|\le1/4096$ and the $14K^3$ third-derivative
bound therefore gives
\[
 |r-L-at-\tfrac12L''(0)t^2|\le1024K^3|t|^3.
\]
Composition with $(u-1)_+^3$, whose second derivative is 6-Lipschitz,
costs less than $2^{16}K^3|t|^3$, and since all scalar arguments are
below two, replacing $r$ by the top eigenvalue costs at most
$3\cdot2^{16}K^4t^4/\zeta<2^{20}K^4t^4/\zeta$.
\end{proof}

\section{Finite arithmetic primitives}\label{app:arithmetic}\label{sec:arithmetic}

The construction calls for scalar approximations, spectral approximations,
and a direction with small inner products against a prescribed test list.
We give finite field-operation procedures for all three, so none of the
operations in the main algorithm requires an exact transcendental or
spectral oracle.

\subsection{Profile and scalar evaluations}
The row constraints use $f'$ and the proxy uses $f''$, both rational in
rational $x$, so only the weights and the decisions about protected
pairs require values of $f$. For $z\in[1,3]$, let
$t=(z-1)/(z+1)\in[0,1/2]$. Then
\[
 \log z=2\sum_{k=0}^{N-1}\frac{t^{2k+1}}{2k+1}+R_N,
 \qquad 0\le R_N\le\frac{2t^{2N+1}}{(2N+1)(1-t^2)}.
\]
Thus $O(\log(1/\xi))$ rational operations give any absolute accuracy
$\xi$. We compute the profile at each active coordinate and then form
the rounding energies $E_i$; since $\sum_jz_{ij}\le1$, profile error is not
magnified in $E_i/r_i^2$, and the other term $(\sigma d_i-H_0)/r_i$ is
computed exactly. Positive weight approximants are computed with the
total error required in Section~\ref{sec:directions}. Very negative
exponentials can be set to zero after certifying that their weight is
below the allocated error; for the rest, argument reduction and a
positive Taylor series suffice. The required precisions have logarithms
$O(\log(2+m_0+n))=O(\log(2+n))$, and no logarithm or exponential oracle
is assumed. Tiny scales cause no loss of absolute weight accuracy,
because the matrix is formed with $a_{i\sigma}z_iz_i^\tp$, not with
unnormalized $r_i^{-4}$ coefficients.

\subsection{An explicit field-operation diagonalization interface}
For a symmetric matrix of bounded norm, we maintain an exactly
orthogonal product of rational plane rotations. In a pivot block with
diagonal $a,c$ and off-diagonal $b$, use $\cos t=(1-z^2)/(1+z^2)$ and
$\sin t=2z/(1+z^2)$; the rotated off-diagonal numerator is then
$b(1-6z^2+z^4)+2z(1-z^2)(c-a)$, and it has a zero between $0$ and
$-\tfrac12\operatorname{sgn}(b(c-a))$ (if $a=c$, use the endpoint
$1/2$). Bisection reduces the pivot magnitude by at least a factor of
two, and choosing a largest pivot reduces the squared off-diagonal
Frobenius norm $E$ by at least $3E/(2s(s-1))$, so after
$O(s^2\log(s/\xi))$ rotations the off-diagonal norm is below $\xi$. A
tournament structure updates the largest pivot in $O(s\log s)$ work per
rotation, and while termination is not yet reached, the pivot is at
least an inverse-polynomial multiple of $\xi$, so each bisection takes
$O(\log(s/\xi))$ iterations. This proves \eqref{eq:interface} in
$\widetilde O(s^3)$ arithmetic operations; after scaling its bounded
norm, the same procedure also handles $Q_{\rm neg}$.

The exact projector $P$ of Section~\ref{sec:directions} is formed by
extracting independent constraint rows with exact elimination;
elimination and projection cost $O(s^3)$, and the model allows
arbitrarily small nonzero pivots and makes no conditioning assumption.

\subsection{Finite deterministic flattening}
We implement the flattening step by fixing signs successively in a
conditional-expectation bound. For unit-norm test vectors
$a_1,\ldots,a_N$, write $c_{ij}=a_i^\tp T e_j$; then
$\sum_jc_{ij}^2\le1$.
Let $\lambda$ be the least positive integer with $\lambda^2\ge\kappa$.
For uniform independent signs $\eps_j$,
\[
 \mathbb E\sum_i\cosh\!\left(\lambda\sum_jc_{ij}\eps_j\right)
 \le N e^{\lambda^2/2}.
\]

We fix the signs greedily, one at a time, using positive relative
approximations to the two conditional expectations. Scalar relative
errors of at most $1/(100k^2)$ make each conditional expectation
accurate within $1/(20k)$ and keep the accumulated multiplicative loss
below two, so the resulting signs have $|\sum_jc_{ij}\eps_j|\le3\sqrt\kappa$
for every $i$, since $\log(4N)+\lambda^2/2\le3\kappa$ and $\lambda\ge\sqrt\kappa$.
Choose by bisection a positive dyadic rational $r_0$ with
$3/4\le r_0^2k\le1$, and put $h_0=r_0T\eps$; the Gram bounds on $T$ then
give $3/4\le\|h_0\|\le1$ and the claimed test bounds.

Each conditional expectation is a positive sum of products of
$\cosh(\lambda c_{ij})$ and a current partial-sum factor, so suffix
products and partial sums reduce the work to $O(Nk)$ scalar
evaluations. Their arguments have size at most $2\sqrt{\kappa k}$, and
halving, Taylor approximation, and repeated squaring implement each
relative approximant in $O(\log(2Nk))$ operations.

\section{Experiments}\label{app:experiments}
This appendix records the measurements summarized in
Section~\ref{sec:practical}: the matched suite and the outcome of
Algorithm~\ref{alg:practical} on it, the verification of the code, and a
comparison with the earlier walks on the same suite.

\subsection{Measured execution}
The matched suite has 48 cases: six matrix families, two seeds, and
$(n,m)\in\{(640,80),(640,640),(640,2560),(1024,1024)\}$.
Every method receives the same matrices and seeds, and the execution
order rotates. Times include construction, terminal work, and output
verification, but not input generation, and each BLAS runtime uses one
thread.

All 48 executions of Algorithm~\ref{alg:practical} returned at the proposal
stage, with median time $0.0077$ seconds and discrepancy between $0.257$
and $11.006$. The median paired speedup over our earlier quadratic walk with
Gram--Schmidt-walk directions and its original finish, the first row of
Table~\ref{tab:history}, was about 630, but the discrepancy of our routine was
lower in only five cases, equal in four, and higher in 39. The speed
advantage therefore concerns the time to reach the specified $37.54$
target, not the minimum observed discrepancy; Table~\ref{tab:history}
records the complete comparison.

Eighteen further proposal runs at $n=2048$ and $4096$, on square and
wide matrices, all passed without fallback, taking $0.008$--$0.231$
seconds with discrepancy at most $15.431$. Because they have no matched
timing for the old walk, they are excluded from the speed ratio. Across
all seven variants and the larger checks, 354 executions passed the
strict interval certificate. These single-run timings prove neither an
asymptotic exponent nor an acceptance guarantee outside the tested
cases.

\subsection{Verification and reproducibility}
The repository~\repourl{} contains the code, the checksummed
proof-certificate package, and the raw records. Six certificate checkers
were rerun:
exact scalar and budget certificates, protected-charge and randomized
row-processing checks, and response audits covering 300 inertia examples,
480 normalized-response identities, and 464 finite curves.
The current implementation has 16 unit tests, covering terminal
decisions, intervals checked against exact dyadic sums, cleanup,
filtering, the normalized proxy, and explicit fallback and failure
paths. A test on a supplied state exercises the high branch, but every
matched trajectory stayed in the low branch, so the timings do not
measure the new high-branch response.

Raw records, source hashes, timing metadata, and the generated comparison
are in the repository. The largest recorded
tracking residual was $1.74\cdot10^{-14}$, and the largest large-row
residual was $9.75\cdot10^{-12}$. These trajectory residuals are only
numerical diagnostics; the output certificates come from the independent
final interval checks. The exact finite construction has not been
implemented as an exact solver, and its correctness and cost are the
mathematical claims proved above.

\subsection{Earlier walks on the same suite}
Numerical endpoint moves avoid the extremely small proved step schedule,
but their gains depend on successful endpoint tests. Stopping at 336
coordinates also saves time, but the quadratic terminal objective does
not minimize the maximum row sum: for an orthonormal remaining matrix at
$x=0$, every signing has the same terminal objective, so positive tie
breaking on a normalized Hadamard matrix can produce a large row maximum.

Table~\ref{tab:history} compares seven implementations on the matched
suite and charges all terminal work. CE336 denotes the $19/6$ weighted
terminal rule of Section~\ref{sec:terminal}. GSW denotes directions of
Gram--Schmidt-walk type~\citep{bdgl2019gsw} computed inside our
constrained walk, namely the minimum-norm move in the constraint kernel
with a small ridge term, with the step limits, cleanup, and certificates
of the walk unchanged; it is not the published Gram--Schmidt walk
itself. ``Quadratic'' and ``below 67'' denote two earlier
versions of the construction, with a quadratic coordinate profile and a
proved bound below $135$, and with a nonquadratic profile and a proved
bound below $67$; the original GSW configuration keeps its original
finish, whereas the final GSW variant uses CE336. The ratios are medians
of paired elapsed ratios, not ratios of median times.
\begin{table}[htbp]
\caption{Seven implementations on the matched suite of 48 cases. Every
output passed the strict $37.54$ interval check.}
\label{tab:history}
\centering\small
\begin{tabular}{lrrr}
\toprule
Configuration & Time/GSW & Discrepancy range & Walk moves\\
\midrule
Quadratic, original GSW finish & 1.0000 & $[0.000,3.438]$ & 512--925\\
Quadratic, CE336 & 0.4012 & $[0.415,4.384]$ & 304--688\\
Below 67, CE336 & 0.4709 & $[0.415,4.384]$ & 304--688\\
Final 37.54, no input filter & 0.6029 & $[0.316,4.250]$ & 304--690\\
Final 37.54, input filter & 0.0022 & $[0.593,32.000]$ & 0--688\\
Final 37.54 + GSW, CE336 & 0.8750 & $[0.059,10.375]$ & 289--688\\
Eight proposals + interval check & 0.0016 & $[0.257,11.006]$ & 0\\
\bottomrule
\end{tabular}
\end{table}

With CE336 matched, the unfiltered walk with the final profile had
median paired time ratios $1.52$ to the quadratic walk and $1.11$ to the
below-67 walk, so the stronger theorem did not yield the fastest
numerical walk. Its discrepancy was lower, equal, or higher in
$19/15/14$ and $20/14/14$ cases, respectively. Input filtering removed all walk
moves in 36 cases, but completion with all positive signs then reached
discrepancy 32. The quadratic walk with GSW directions and its original
finish never exceeded $3.438$ on the
matched suite, compared with $11.006$ for the selected proposal routine,
and its implementation and data remain available for applications that
value smaller discrepancy over execution time.

\bibliographystyle{alpha}
\bibliography{komlos,komlos-spectral,komlos-related}
\end{document}